\documentclass[hidelinks,onefignum,onetabnum]{siamart251216}

\usepackage{lipsum}
\usepackage{amsfonts}
\usepackage{graphicx}
\usepackage{epstopdf}
\usepackage{algorithmic}
\ifpdf
  \DeclareGraphicsExtensions{.eps,.pdf,.png,.jpg}
\else
  \DeclareGraphicsExtensions{.eps}
\fi

\newsiamremark{remark}{Remark}
\newsiamremark{hypothesis}{Hypothesis}
\crefname{hypothesis}{Hypothesis}{Hypotheses}
\newsiamthm{claim}{Claim}
\newsiamremark{fact}{Fact}
\crefname{fact}{Fact}{Facts}

\headers{AP DPM for collisional plasma}{Y. Huang and L. Wang}

\title{Asymptotic-preserving deterministic particle methods for collisional plasma models\thanks{Submitted to the editors DATE.
\funding{This work is partially supported by NSF grant DMS-2513336.}}}

\author{Yan Huang\thanks{School of Mathematics, University of Minnesota, Minneapolis, MN, 55455 
  (\email{huan2728@umn.edu}, \email{liwang@umn.edu}).}
\and Li Wang\footnotemark[2]
}

\usepackage{amsopn}

\usepackage{hyperref}
\usepackage{color}
\usepackage{enumerate}

\allowdisplaybreaks[4]

\newcommand{\rd}{\mathrm{d}}
\newcommand{\R}{\mathbb{R}}
\newcommand{\T}{\mathsf{T}}
\newcommand{\bx}{\boldsymbol{x}}
\newcommand{\bv}{{\boldsymbol{v}}}
\newcommand{\bsv}{{\boldsymbol{\mathsf{v}}}}

\newcommand{\bz}{\boldsymbol{z}}

\newcommand{\bu}{\boldsymbol{u}}
\newcommand{\bs}{\boldsymbol{s}}

\newcommand{\M}{\mathcal{M}}

\newcommand{\tr}{\operatorname{Tr}}
\newcommand{\cof}{\operatorname{cof}}
\newcommand{\Qop}{\mathcal{Q}}
\newcommand{\bxi}{\boldsymbol{\xi}}

\begin{document}

\maketitle

\begin{abstract}
   We develop novel asymptotic-preserving (AP) deterministic particle methods for collisional plasma models, including both Landau–Fokker–Planck and Dougherty collision operators, under hydrodynamic scaling. Our schemes treat the non-stiff transport part explicitly and the stiff collision operators fully implicitly through the energy-conserving Jordan--Kinderlehrer--Otto (JKO) schemes by exploiting their gradient flow structures. This approach extends our previous work on the space-homogeneous Landau equation \cite{huang2024jkolandau} and introduces a new treatment of the Dougherty operator via a projected gradient flow formulation. We identify the crucial role of Jacobian log-determinant evaluation in stiff regimes and introduce an inner-time quadrature strategy that improves both accuracy and efficiency. Furthermore, we uncover intriguing connections with score-based transport modeling, showing that both explicit and implicit score matching arise as special cases of our unified variational framework and exhibit limitations in the stiff regime. We also develop practical large-scale implementations via neural network parameterization and efficient training strategies. Various numerical examples demonstrate the structure-preserving and AP properties of our schemes for general initial data.
\end{abstract}

\begin{keywords}
Landau operator, Dougherty collision, asymptotic-preserving scheme, deterministic particle method, JKO scheme, neural network, explicit/implicit score-matching 
\end{keywords}

\begin{MSCcodes}
    65M75, 82C40, 82D10, 82M30, 68T07
\end{MSCcodes}

\section{Introduction}
Kinetic equations provide a fundamental and accurate description of plasma dynamics, capturing velocity-space physics, collective effects, and non-equilibrium behavior beyond the reach of fluid models. In this work, we focus on the following kinetic equation:
\begin{equation} \label{eqn0}
    \partial_t f + \bv \cdot \nabla_{\bx} f = \frac{1}{\varepsilon} \mathcal{Q}(f) \,,~~ \bx \in \Omega \subseteq \R^{d_x} \,, \bv \in \R^{d_v} \,,
\end{equation}
where $f = f(t, \bx, \bv)$ denotes the particle density at time $t$, position $\bx$, and velocity $\bv$. The operator $\mathcal{Q}(f)$ models particle interactions. The parameter $\varepsilon$, known as the Knudsen number, is a dimensionless quantity defined as the ratio of the mean free path to a characteristic length scale, and it characterizes the strength of the collisional effects.

Two forms of $\mathcal{Q}(f)$ are commonly used to describe the long-range interactions among charged particles in plasmas: 
\begin{itemize}
    \item[(i)] The Landau–Fokker–Planck operator (or simply the Landau operator), derived by Lev Landau in 1936 \cite{landau1958kinetic}, takes the form 
\begin{equation*}
    \mathcal{Q}_L(f,f) = \nabla_{\bv} \cdot \int_{\R^{d_v}} A(\bv-\bv_*) (f(\bv_*)\nabla_{\bv} f(\bv) - f(\bv)\nabla_{\bv_*} f(\bv_*)) \rd\bv_* \,,
\end{equation*}
where the collision kernel is
$$A(\bz) = |\bz|^{\gamma + 2} \left( I_{d_v} - \frac{\bz \otimes \bz}{|\bz|^2} \right) =: |\bz|^{\gamma + 2} \Pi(\bz) \,.$$
The matrix $A(\bz)$ is symmetric and positive semi-definite, with null space $\operatorname{span}\{\bz\}$. The parameter $\gamma$ depends on the interaction law: for inverse power-law forces proportional to $r^{-s}$, one has $\gamma = \frac{s-5}{s-1}$. In particular, $\gamma>0$ corresponds to hard potentials, $\gamma=0$ to Maxwellian molecules, and $\gamma<0$ to soft potentials. The Coulomb interaction relevant to plasma physics corresponds to $\gamma=-3$ (i.e., $s=2$) \cite{landau1958kinetic}.

\item[(ii)] The Dougherty operator \cite{DFP}, also known as the Lenard–Bernstein operator \cite{PhysRev.112.1456}, is defined as
\begin{equation*}
\mathcal{Q}_D(f) = T \Delta_{\bv} f + \nabla_\bv \cdot \big((\bv - \bu) f \big) \,,
\end{equation*}
where $T$ and $\bu$ denote the local temperature and mean velocity, defined by
\begin{equation} \label{rhouT}
    \rho = \int_{\R^{d_v}} \!\!f(\bx, \bv) \rd\bv \,,~ \bu = \frac{1}{\rho}\int_{\R^{d_v}} \!\bv f(\bx, \bv) \rd\bv \,,~ T = \frac{1}{d_v \rho}\int_{\R^{d_v}} \!\!\!|\bv-\bu|^2 f(\bx, \bv) \rd\bv \,.
\end{equation}
The Dougherty operator was introduced as a simplified local approximation to the Landau operator $\mathcal{Q}_L$, avoiding the explicit nonlocal binary-interaction structure.
\end{itemize}

In both cases, the collision operators share key structural properties, including the conservation of mass, momentum, and energy, as well as the dissipation of entropy. In particular, the Landau operators admit the following logarithmic form:
\begin{flalign}
    \Qop_L(f,f) &= \nabla_{\bv} \!\cdot\! \left(f(\bv)\!\! \int_{\R^{d_v}}\!\! A(\bv-\bv_*) \bigl(\nabla_{\bv}\log f(\bv) - \nabla_{\bv_*} \log f(\bv_*)\bigr) f(\bv_*) \rd\bv_* \right). \label{LD_op}
\end{flalign}
For a suitable test function $\varphi=\varphi(\bv)$, their weak formulations are given by
\begin{flalign*}
    \int_{\R^{d_v}} \Qop_L(f,f)\varphi \,\rd\bv
    &= -\frac{1}{2} \int_{\R^{2d_v}} \bigl(\nabla_{\bv}\varphi(\bv) - \nabla_{\bv_*}\varphi(\bv_*)\bigr) \cdot A(\bv-\bv_*) \\
    & \hspace{1.5cm} \bigl(\nabla_{\bv}\log f(\bv) - \nabla_{\bv_*}\log f(\bv_*)\bigr) f(\bv_*) f(\bv)\,\rd\bv_* \rd\bv \,.
\end{flalign*}
By choosing $\varphi=1$, $\bv$, and $|\bv|^2$, one recovers the conservation of mass, momentum, and energy, respectively. Moreover, taking $\varphi=\log f$ and using the symmetry and non-negativity of $A$,
one obtains the entropy dissipation inequality
\[\int_{\R^{d_v}} \Qop_L(f,f)\log f \,\rd\bv \le 0 \,.\]
Equality holds if and only if $f$ is a local Maxwellian equilibrium of the form:
\begin{equation}\label{local_eq}
    \M_{U}(\bv) = \frac{\rho}{(2\pi T)^{{d_v}/2}} \exp\left(-\frac{|\bv-\bu|^2}{2T}\right),
\end{equation}
where $U = (\rho, \bu, T)$ is defined in \eqref{rhouT}. 

On the other hand, for the Dougherty operator, conservation follows by direct calculation. Moreover, it can be rewritten in the form:
\begin{align}\label{Dou_op}
     \Qop_D(f) = T\nabla_\bv \!\cdot\! \left(f(\bv) \nabla_\bv \log \tfrac{f(\bv)}{\mathcal{M}_U(\bv)} \right) , 
\end{align}
which then admits the weak formulation
\begin{align*}
    \int_{\R^{d_v}} \Qop_D(f)\varphi \,\rd\bv
    = -T \int_{\R^{d_v}} \nabla_\bv \varphi(\bv) \cdot \nabla_\bv \log \tfrac{f(\bv)}{\mathcal{M}_U(\bv)}\, f(\bv)\,\rd\bv \,.
\end{align*}
Taking $\varphi=\log \tfrac{f}{\mathcal{M}_U}$ together with the conservation properties, one obtains the entropy dissipation inequalities
\[\int_{\R^{d_v}} \Qop_D(f)\log f \,\rd\bv \le 0 \,.\] 

As the Knudsen number $\varepsilon \to 0$, the solution $f$ of \eqref{eqn0} converges to the local Maxwellian \eqref{local_eq}, where the macroscopic quantities $\rho$, $\bu$, and $T$ satisfy the compressible Euler system:
\begin{align}\label{Euler}
\begin{cases}
    \partial_t \rho + \nabla_{\bx} \cdot (\rho \bu) = 0 \,, \\
    \partial_t(\rho \bu) + \nabla_{\bx} \cdot (\rho \bu \otimes \bu + p I_{d_x}) = 0 \,, \\
    \partial_t E + \nabla_{\bx} \cdot ((E + p)\bu) = 0 \,.
\end{cases}
\end{align}
Here $E = \frac{1}{2}\rho |\bu|^2 + \frac{d_v}{2}\rho T$ denotes the total energy, and the pressure $p$ is given by the constitutive relation $p = \rho T$. Accordingly, the regimes $\varepsilon \ll 1$ and $\varepsilon = \mathcal{O}(1)$ are referred to as the fluid regime and the kinetic regime, respectively.

Numerical discretization of the Landau operator is challenging due to its high dimensionality, nonlocal interactions, and the need to preserve key physical structures. Existing approaches include grid-based methods, such as finite-difference entropy schemes \cite{degond1994entropy, buet1998conservative}, fast spectral methods \cite{fourier_spectral_landau}, and formulations based on Rosenbluth potentials \cite{TAITANO2015357}; as well as particle-based methods, including stochastic Monte Carlo approaches \cite{dimarco2010direct, ROSIN2014140, DU2025114387}, hybrid methods \cite{Caflisch2008}, and deterministic particle methods \cite{carrillo2020particle, Carrillo_Jin_Tang_2022}. More recently, learning-enhanced particle methods have also been proposed \cite{HUANG2025114053, huang2024jkolandau, ilin2024transportbasedparticlemethods}.

When spatial dependence is included, especially in the presence of the multiscale parameter $\varepsilon$, the problem becomes even more challenging. In particular, it is desirable to design methods that perform uniformly well across regimes with a fixed computational budget. This requirement falls within the framework of asymptotic-preserving (AP) schemes \cite{jin1999efficient}, which have been extensively studied for various forms of \eqref{eqn0}; see \cite{Jin_2022} for a recent review. When $\mathcal Q$ takes the form of the Landau operator, however, existing AP methods are predominantly grid-based and typically rely on implicit treatment of the collision term \cite{TAITANO2015357}, penalization \cite{JIN20116420}, micro–macro decompositions \cite{GAMBA2019264}, or exponential integrators \cite{Dimarco_Li_Pareschi_Yan_2015}.

To maintain compatibility with prevailing methods for collisionless plasma simulation, the particle-in-cell (PIC) methods \cite{Birdsall2018}, we aim to develop a particle-based AP scheme in this work, which also benefits from improved scalability in high dimensions. While particle-based AP methods do exist, they are mostly designed for short-range, Boltzmann-type collisions. Only recently have AP approaches for Landau-type collisions been proposed \cite{CHEN2025113771, medaglia2025dsmcpic}. However, these methods are typically stochastic or hybrid in nature, and thus suffer from statistical noise and do not explicitly enforce structure preservation.

In this paper, we develop deterministic particle-based AP methods for both Landau and Dougherty collisions. Our approach builds on our previous work \cite{huang2024jkolandau}, exploiting the gradient flow structure of the {\it stiff} collision operators and treating them {\it fully implicitly} via the minimizing movement scheme, also known as the Jordan–Kinderlehrer–Otto (JKO) scheme \cite{jko}. A major difficulty, however, arises in the regime $\varepsilon \ll 1$, where the resulting optimization problem may fail to converge or converge to incorrect solutions if the discretization is not handled carefully and the optimizer is applied naively. We show that accurate evaluation of the Jacobian log-determinant of the transport map is crucial for success, which necessitates high-order inner-time discretization together with high-order quadrature. We also discuss extensions of our space-homogeneous formulation \cite{huang2024jkolandau} to the spatially dependent setting, either via operator splitting or through an additional regularization to delocalize the Dirac Delta function. In practice, we adopt the splitting approach due to its robustness and ease of implementation.

The main contributions of this work are as follows.
\begin{enumerate}[(i)]
    \item We propose deterministic particle-based asymptotic preserving schemes for both Landau and Dougherty collisions. The non-stiff transport part is treated explicitly, while the stiff collision term is handled implicitly via a dynamic JKO formulation. This approach extends our previous work on the space-homogeneous Landau equation \cite{huang2024jkolandau} and introduces a new treatment of the Dougherty operator through a projected gradient flow formulation. 
    \item We identify the key role of Jacobian log-determinant evaluation in stiff regimes and introduce an inner-time quadrature strategy that improves both accuracy and efficiency. We also highlight intriguing connections with score-based transport modeling. In particular, we show that both explicit and implicit score matching arise as special cases of our unified variational framework; however, both exhibit limitations in the stiff regime, as the former is unstable while the latter (with a fixed point implementation) may converge to an incorrect optimizer.
    \item We develop practical large-scale implementations via neural network parametrization and efficient training strategies.
\end{enumerate}

The rest of the paper is organized as follows. Section~\ref{sec:2} presents the variational formulations for Landau and Dougherty equations, together with optimization challenges and remedies. Section \ref{sec:3} details particle discretization, neural network parametrization, inner-time discretization, and efficient training strategies. Section \ref{sec:4} reports numerical experiments. The paper concludes in Section~\ref{sec:5}.

\section{Variational formulation and optimization}\label{sec:2}
This section presents the variational framework for computing solutions to the Landau and Dougherty equations. We examine the challenges that arise as $\varepsilon \to 0$ and highlight the necessity of accurately evaluating the Jacobian determinant of the transport map. We also discuss connections to score-based approaches, and emphasize the advantages of our variational framework for handling stiff problems.

\subsection{Dynamic JKO scheme for the collision term}
In this subsection, we derive dynamic JKO formulations for both homogeneous Landau and Dougherty equations, i.e., 
\begin{align} \label{homo}
    \partial_t f(t,\bv) = \frac{1}{\varepsilon}\mathcal Q (f)\,,
\end{align}
which serve as the core implicit collision solvers in our AP framework.  The treatment of the Landau operator largely follows our previous work \cite{huang2024jkolandau}, while for the Dougherty collision operator we present two formulations: a standard Wasserstein gradient flow (WGF) formulation and a conservative projected-WGF formulation, the latter offering significantly improved conservation properties in practical implementations.

\subsubsection{Landau collision}
We first recall the variational formulation of the homogeneous Landau equation \eqref{LD_op}. As shown in \cite{carrillo2024landau}, it is the gradient flow of the entropy functional $\mathcal{H}(f) = \int_{\R^{d_v}} f \log f \rd\bv$ with respect to the following Landau metric $d_L$, defined on the space of probability measures:
\begin{flalign*}
    d_L^2(f_0, f_1) &:= \inf_{f, \bs} ~ \int_0^1 \frac{1}{2} \int_{\R^{2d_v}} |\bs(\tau, \bv) - \bs(\tau, \bv_*)|^2_{A(\bv - \bv_*)} f(\tau, \bv)f(\tau, \bv_*) \rd\bv \rd\bv_* \rd\tau \,, \\ 
    & s.t. ~ \partial_\tau f + \nabla_{\bv} \!\cdot\! \left(f \int_{\R^{d_v}} A(\bv \!-\! \bv_*) (\bs(\tau, \bv) \!-\! \bs(\tau, \bv_*)) f(\tau, \bv_*) \rd\bv_* \right) = 0 \,, \\
    & ~~~~~ f(0, \cdot)=f_0 \,,~ f(1, \cdot)=f_1 \,,
\end{flalign*}
where $|\bz|^2_A := \bz^\top A \bz$. Building on this result, we proposed a dynamic JKO scheme for solving the homogeneous Landau equation in \cite{huang2024jkolandau}, extending previous works for Wasserstein-type gradient flows \cite{Carrillo2022primaldual, LI2020109449, Carrillo2024primaldual}. Specifically, given $f^n$, we obtain $f^{n+1}:= f(1, \cdot)$ with $f(\tau, \bv)$ solving 
\begin{equation}\label{LD_DJKO}
\begin{split}
    & \inf_{f,\bs} \varepsilon \int_0^1 \frac12 \int_{\R^{2d_v}} |\bs(\tau, \bv)-\bs(\tau, \bv_*)|^2_{A(\bv-\bv_*)} f(\tau, \bv) f(\tau, \bv_*) \rd\bv \rd\bv_* \rd\tau + 2\Delta t \mathcal{H}(f(1, \cdot)) \,, \\
    & s.t.~ \partial_\tau f + \nabla_{\bv} \!\cdot\! \left(f \int_{\R^{d_v}} A(\bv-\bv_*) (\bs(\tau, \bv)-\bs(\tau, \bv_*))f(\tau, \bv_*) \rd\bv_* \right) = 0 \,,\ f(0, \cdot) = f^n \,.
\end{split}
\end{equation}
The dynamic JKO scheme \eqref{LD_DJKO} preserves all desirable properties--namely, the local conservation of mass, momentum, and energy, as well as the local entropy dissipation property \cite[Proposition 2]{huang2024jkolandau}. In Lagrangian coordinates, \eqref{LD_DJKO} can be rewritten as: we obtain $f^{n+1} := {\T_1}{\sharp}f^n$ by solving
\begin{equation}\label{LD_LJKO}
\begin{split}
    & \inf_{\bs} ~ \varepsilon \int_0^1 \frac{1}{2} \int_{\R^{2d_v}} |\bs(\tau, \T_\tau(\bv)) - \bs(\tau, \T_\tau(\bv_*))|^2_{A(\T_\tau(\bv) - \T_\tau(\bv_*))} f^n(\bv) f^n(\bv_*) \rd\bv \rd\bv_* \rd\tau \\ 
    & \qquad + 2\Delta t \mathcal{H}({\T_1}{\sharp} f^n) \,, \\
    & s.t. ~
    \tfrac{\rd}{\rd\tau} \T_\tau(\bv) = \int_{\R^{d_v}} A(\T_\tau(\bv) - \T_\tau(\bv_*)) [\bs(\tau, \T_\tau(\bv)) - \bs(\tau, \T_\tau(\bv_*))] f^n(\bv_*) \rd\bv_* \,, \\
    & ~~~~~~ \T_0(\bv)=\bv \,.
\end{split}
\end{equation}

\subsubsection{Dougherty collision}
Below we present two variational formulations for the one-step update of the homogeneous Dougherty equation, given $f^n (\bv)$.

\paragraph{Wasserstein gradient flow}
Recall the form \eqref{Dou_op} of the Dougherty operator, and notice that the macroscopic quantities $U = (\rho, \bu, T)$ defined in \eqref{rhouT} remain unchanged along the dynamics \eqref{homo}, we denote $U^n = (\rho^n, \bu^n, T^n)$ computed from $f^n$. Then the homogeneous Dougherty can be written as a Wasserstein gradient flow: 
\begin{equation} \label{WGF}
    \partial_t f =  -\tfrac{T^n}{\varepsilon} \nabla_{W_2} \mathcal H(f|\mathcal{M}_{U^n}) =  \tfrac{T^n}{\varepsilon} \nabla_\bv \cdot \left( f \nabla_\bv \tfrac{\delta \mathcal{H}(f \mid \mathcal{M}_{U^n})}{\delta f} \right) ,\ f(0, \cdot) = f^n \,.
\end{equation}
The corresponding dynamic JKO scheme is as follows \cite{LEE2024113187}: solve  $f(\tau, \bv)$ from 
\begin{equation}\label{WGF_JKO}
\begin{split}
    & \inf_{f, \bs} ~ \varepsilon \int_0^1 \!\! \int_{\R^{d_v}} |\bs(\tau, \bv)|^2 f(\tau, \bv) \rd\bv \rd\tau + 2\Delta t T^n \mathcal{H}\left(f(1,\cdot) \mid \M_{U^n} \right) , \\
    & s.t. ~~ \partial_\tau f + \nabla_{\bv} \cdot (f\bs)= 0 \,,\ f(0, \cdot) = f^n \,,
\end{split}
\end{equation}
and then set $f^{n+1}(\bv) = f(1,\bv)$.

\paragraph{Projected Wasserstein gradient flow}
It is important to note that while \eqref{WGF} conserves mass, momentum, and energy, its semi-discrete counterpart \eqref{WGF_JKO} preserves mass and momentum only at the exact minimizer; this property may be lost if the optimization is not carried out exactly.
The energy $E^n = \frac{1}{2}\int_{\R^{d_v}} |\bv|^2 f^n \rd\bv$ is dissipating rather than conserved, even at the exact minimizer. See the discussion in Appendix~\ref{apdx:A}. This stands in sharp contrast to \eqref{LD_LJKO}, where all conservation is directly built into the constraint ODE dynamics. To restore these properties, we reinterpret the Dougherty equation as a projected Wasserstein gradient flow on a submanifold with fixed momentum and energy.

More specifically, consider the space of probability measures $\mathcal{P}_2$ endowed with the Wasserstein metric. We define the following subspace \(\mathcal{P}_2^{con}\):
\begin{equation*}
    \mathcal{P}_2^{con} = \left\{ f : f \in \mathcal{P}_2  ~~\text{and}~ \int_{\R^{d_v}} \bv f \rd \bv = \rho^n\bu^n, ~ \int_{\R^{d_v}} |\bv|^2 f \rd \bv = 2E^n \right\} .
\end{equation*}
Following Otto's calculus \cite{villani2003topics}, for any $f\in  \mathcal{P}_2^{con} $, the tangent space $\mathcal{T}_f\mathcal{P}_2^{con}$  contains admissible perturbations that preserve momentum and energy:
\begin{equation*}
    \mathcal{T}_f\mathcal{P}_2^{con} = \left\{ - \nabla_{\bv} \cdot(f \nabla_{\bv}\phi) : \nabla_{\bv}\phi \in L^2_f \,, \int_{\R^{d_v}} \nabla_{\bv}\phi f \rd\bv = 0 \,, \int_{\R^{d_v}} \bv \cdot \nabla_{\bv}\phi f \rd\bv = 0 \right\} .
\end{equation*}
Consequently, for any functional $\mathcal H: \mathcal{P}_2^{con} \mapsto \mathbb{R}$, its gradient with respect to the Wasserstein metric restricted to this subspace is given by the projection of the standard Wasserstein gradient $\nabla_{W_2} \mathcal H(f)$ onto $\mathcal{T}_f\mathcal{P}_2^{con}$. Specifically, for the entropy functional $\mathcal H(f) = \int_{\R^{d_v}} f\log f \rd \bv$, this projected gradient is
\begin{equation*}
    \operatorname{Proj}_f (\nabla_{W_2} \mathcal{H}(f)) = -\nabla_{\bv} \cdot (f \nabla_{\bv} (\log f + \tfrac{\bv-\bu}{T})) \,.
\end{equation*}
More detailed calculation is deferred to Appendix~\ref{apdx:B}. As a result, the Dougherty equation can be realized as 
\begin{equation*}
    \partial_t f = \tfrac{T^n}{\varepsilon} \nabla_\bv \cdot \left( f \left(\nabla_\bv \log f + \tfrac{\bv-\bu}{T} \right) \right) = -\tfrac{T^n}{\varepsilon} \operatorname{Proj}_f (\nabla_{W_2} \mathcal{H}(f)) \,.
\end{equation*}

This new viewpoint leads to the following conservative dynamic JKO scheme: given $f^n$, we obtain $f^{n+1}:= f(1, \cdot)$ with $f(\tau, \bv)$ solving
\begin{equation*}
\begin{split}
    & \inf_{f,\bs} ~ \varepsilon \int_0^1 \!\! \int_{\R^{d_v}} |\bs(\tau, \bv)|^2 f(\tau, \bv) \rd\bv \rd\tau + 2\Delta t T^n \mathcal{H}(f(1, \cdot)) \,, \\
    & s.t.~~ \partial_\tau f + \nabla_{\bv} \cdot (f \bs) = 0 \,,\ \int_{\R^{d_v}} \bs f \rd\bv = 0 \,,\ \int_{\R^{d_v}} \bv \cdot \bs f \rd\bv = 0 \,,\ f(0, \cdot) = f^n \,. 
\end{split}
\end{equation*}
In practice, the constraints on $\bs$ are enforced by replacing it with its projected component \( \bs^\perp\):
\begin{equation*}
    \bs^{\perp}(\tau, \bv) \!:=\! \bs(\tau, \bv) - \frac{\int_{\R^{d_v}} \bs(\tau, \bv) f(\tau, \bv) \rd\bv}{\int_{\R^{d_v}} f(\tau, \bv) \rd\bv} - \frac{\int_{\R^{d_v}} \bs(\tau, \bv) \!\cdot\! (\bv \!-\! \bu(\tau)) f(\tau, \bv) \rd\bv}{\int_{\R^{d_v}} |\bv \!-\! \bu(\tau)|^2 f(\tau, \bv) \rd\bv} (\bv - \bu(\tau)) \,.
\end{equation*}
Correspondingly, the Lagrangian form is written as: we obtain $f^{n+1}:= {\T_1}{\sharp} f^n$ by solving
\begin{equation}\label{Dou_LJKO}
\begin{split}
    & \inf_{\bs} ~ \varepsilon \int_0^1 \!\! \int_{\R^{d_v}} |\bs^{\perp}(\tau, \T_\tau(\bv))|^2 f^n(\bv) \rd\bv \rd\tau + 2\Delta t T^n \mathcal{H}({\T_1}{\sharp} f^n) \,, \\
    & s.t. ~~ \tfrac{\rd}{\rd\tau} \T_\tau(\bv) = \bs^{\perp}(\tau, \T_\tau(\bv)) \,,\ \T_0(\bv) = \bv \,.
\end{split}
\end{equation}
Here, the Lagrangian form of \(\bs^\perp\) is given by
\begin{equation}\label{s_perp}
\begin{split}
    \bs^{\perp}(\tau, \T_\tau(\bv)) := ~
    & \bs(\tau, \T_\tau(\bv)) - \frac{1}{\rho^n} \int_{\R^{d_v}} \bs(\tau, \T_\tau(\bv)) f^n(\bv) \rd\bv \\
    & - \frac{\int_{\R^{d_v}} \bs(\tau, \T_\tau(\bv)) \cdot (\T_\tau(\bv) - \bu(\tau)) f^n(\bv) \rd\bv}{\int_{\R^{d_v}} |\T_\tau(\bv) - \bu(\tau)|^2 f^n(\bv) \rd\bv} (\T_\tau(\bv) - \bu(\tau)) \,,
\end{split}
\end{equation}
where $\bu(\tau) = \int_{\R^{d_v}} \T_\tau(\bv) f^n(\bv) \rd\bv$.

\subsubsection{Density and entropy computation}
To proceed \eqref{LD_LJKO} and \eqref{Dou_LJKO}, one must evaluate the terminal energy $\mathcal{H}({\T_1}{\sharp} f^n)$. This is accomplished by invoking the relation:
\begin{equation*}
    ({\T_1}{\sharp} f^n)(\T_1(\bv)) = \frac{f^n(\bv)}{|\det \nabla_{\bv} \T_1(\bv)|} \,.
\end{equation*}
Consequently, the entropy can be computed as 
\begin{equation} \label{0402}
    \mathcal{H}(\T_1 \sharp f^n) = \int_{\R^{d_v}} f^n(\bv) \log f^n(\bv) \rd\bv - \int_{\R^{d_v}} f^n(\bv) \log |\det \nabla_{\bv} \T_1(\bv)| \rd\bv \,.
\end{equation}

To efficiently compute the determinant of the Jacobian $|\det \nabla_{\bv} \T_1(\bv)|$, we use the instantaneous change of variable formula \cite{neuralode, grathwohl2018scalable, LEE2024113187},  which describes the evolution of $|\det \nabla_{\bv} \T_\tau(\bv)|$. Specifically, we have
\\ 
\noindent\textit{Landau} \cite[Proposition 4]{HUANG2025114053}:
\begin{flalign*}
    & \frac{\rd}{\rd \tau} \log |\det \nabla_{\bv} \T_\tau(\bv)| = \\
    & \qquad \int_{\R^{d_v}} \nabla_{\bsv} \!\cdot\! \big\{ A(\bsv - \T_\tau(\bv_*)) [\bs(\tau, \bsv) - \bs(\tau, \T_\tau(\bv_*))] \big\} \big|_{\bsv=\T_\tau(\bv)} f^n(\bv_*) \rd\bv_* \,.
\end{flalign*}
\noindent\textit{Dougherty}:
\begin{flalign*}
    & \frac{\rd}{\rd \tau} \log |\det\nabla_{\bv} \T_\tau(\bv)| = \\ 
    & \qquad \nabla_{\bsv} \!\cdot\! \bs(\tau, \bsv) \big|_{\bsv=\T_\tau(\bv)} - 
    d_v \frac{\int_{\R^{d_v}} \bs(\tau, \T_\tau(\bv)) \cdot (\T_\tau(\bv) - \bu(\tau)) f^n(\bv) \rd\bv}{\int_{\R^{d_v}} |\T_\tau(\bv) - \bu(\tau)|^2 f^n(\bv) \rd\bv} \,, 
\end{flalign*}
with the initial condition $\log |\det \nabla_{\bv} \T_0(\bv)| = 0$.

\subsection{Pitfalls and remedies for accurate optimization}\label{sec:2.2}
\eqref{LD_LJKO} and \eqref{Dou_LJKO} present us an optimization problem that requires careful examination when $\varepsilon \ll 1$.  At first glance, the quadratic term arising from the transport metric---responsible for the favorable convexity properties---is scaled by $\varepsilon$, thereby weakening the overall convexity of the objective function and potentially causing difficulties for standard gradient descent methods.

However, convexity alone does not fully explain the challenge. A more fundamental issue is that small $\varepsilon$ can significantly degrade the accuracy of evaluating the objective function, particularly the entropy term, which involves computing a log-determinant and becomes increasingly dominant as $\varepsilon \rightarrow 0$. This section highlights the necessity of accurately computing this log-determinant component and proposes a practical approach to address it. Interesting connections with score-based approaches \cite{boffi2206probability, lu2024score, HUANG2025114053, ilin2024transportbasedparticlemethods} are also highlighted. 

For clarity of exposition, we illustrate our ideas using the simple heat equation
\begin{equation} \label{heat}
    \partial_t f = \frac{1}{\varepsilon}\Delta_{\bv} f\,, \quad f(0,\cdot) = f^n \,,
\end{equation}
rather than the Landau or Dougherty equations, whose additional complexities would obscure the key issues associated with small $\varepsilon$. Once the main ideas are established, we will extend the approach to the Landau and Dougherty settings. 

Viewing \eqref{heat} as the Wasserstein gradient flow of the entropy functional $\mathcal H(f) = \int_{\R^{d_v}} f \log f \rd \bv$, its Lagrangian dynamic JKO scheme then reads: 
\begin{flalign*}
    & \inf_{\bs} ~ \varepsilon \int_0^1 \!\! \int_{\R^{d_v}} |\bs(\tau, \T_\tau(\bv))|^2 f^n(\bv) \rd\bv \rd\tau - 2\Delta t \int_{\R^{d_v}} \log|\det \nabla_{\bv} \T_1(\bv)| f^n(\bv) \rd\bv \,, \\
    & s.t. ~~ \tfrac{\rd}{\rd\tau} \T_\tau(\bv) = \bs(\tau, \T_\tau(\bv)) \,,\ \T_0(\bv) = \bv \,, \\
    & \qquad \tfrac{\rd}{\rd\tau} \log|\det \nabla_{\bv} \T_\tau(\bv)| = \nabla_{\bsv} \!\cdot\! \bs(\tau, \bsv) \big|_{\bsv=\T_\tau(\bv)} \,,\ \log|\det \nabla_{\bv} \T_0(\bv)| = 0 \,.
\end{flalign*}
Integrating the second ODE for the log-determinant and substituting it into the objective function to eliminate the second term, we obtain
\begin{equation}\label{heat_jko_div}
\begin{aligned}
    & \inf_{\bs} ~ \int_0^1 \!\! \int_{\R^{d_v}} \left[\varepsilon |\bs(\tau, \T_\tau(\bv))|^2 - 2 \Delta t \nabla_{\bsv} \!\cdot\! \bs(\tau, \bsv) \big|_{\bsv=\T_\tau(\bv)} \right] f^n(\bv) \rd\bv \rd\tau \,, \\
    & s.t. ~~ \tfrac{\rd}{\rd\tau} \T_\tau(\bv) = \bs(\tau, \T_\tau(\bv)) \,,\ \T_0(\bv) = \bv \,.
\end{aligned}
\end{equation}
We now discuss different choices of inner-time (i.e., $\tau$) discretization and integration. 

\paragraph{One step explicit discretization is unstable} If we discretize the constraint ODE in \eqref{heat_jko_div} using a single forward Euler step, and approximate the time integral in the objective by a left-endpoint quadrature rule, we obtain the minimization problem
\begin{equation} \label{esm0}
    \inf_{\bs} \int_{\R^{d_v}} \left[\varepsilon |\bs(\bv)|^2 - 2 \Delta t \nabla_{\bv} \!\cdot\! \bs(\bv)\right] f^n(\bv) \rd\bv \,.
\end{equation}
Interestingly, after rescaling $\bs \mapsto -\frac{\Delta t}{\varepsilon} \bs$, \eqref{esm0} reduces exactly to the classical (explicit) score matching loss \eqref{esm} \cite{hyvarinen05a}:
\begin{equation}\label{esm}
    \inf_{\bs} \int_{\R^{d_v}} \left[|\bs(\bv)|^2 + 2 \nabla_{\bv} \!\cdot\! \bs(\bv)\right] f^n(\bv) \rd\bv \,.
\end{equation}
Consequently, the optimizer of \eqref{esm0} corresponds to the score at the current step, namely $\bs= -\frac{\Delta t}{\varepsilon}\nabla\log f^n$. When $\varepsilon \ll 1$, this results in an unstable particle update.

\paragraph{One step implicit discretization may lead to a wrong optimizer}
Since the ODE dynamic in \eqref{heat_jko_div} is stiff, it is tempting to adopt an implicit discretization. Coupled with the right-endpoint quadrature rule for the objective function, this leads to the following constrained optimization problem:
\begin{equation}\label{ism}
\begin{aligned}
    & \inf_{\bs} ~ \int_{\R^{d_v}} [\varepsilon |\bs(\T(\bv))|^2 - 2\Delta t (\nabla \!\cdot\! \bs)(\T(\bv))] f^n(\bv) \rd\bv \,, \\
    & s.t. ~~ \T(\bv) = \bv + \bs(\T(\bv)) \,.
\end{aligned}
\end{equation}
While \eqref{ism} may appear reasonable at first glance, as one might expect that the optimizer would be $\bs = -\frac{\Delta t}{\varepsilon} \nabla \log (\T \sharp f^n) = -\frac{\Delta t}{\varepsilon} \nabla \log f^{n+1}$, we show below that this intuition is incorrect. The resulting problem is not only difficult to optimize, but also fails to yield the desired implicit particle update.

To see this, differentiating the constraint, we have 
\begin{equation*}
    \nabla_{\bv} \T(\bv) - I = (\nabla \bs)(\T(\bv)) \nabla_{\bv} \T(\bv) \,,
\end{equation*}
which yields $(\nabla \bs)(\T(\bv)) = I - (\nabla_{\bv} \T(\bv))^{-1}$, and consequently
\begin{equation*}
    (\nabla \cdot \bs)(\T(\bv)) = d_v - \tr\left((\nabla_{\bv} \T(\bv))^{-1}\right).
\end{equation*}
Substituting this expression into \eqref{ism} and rewriting $\bs(\T(\bv))$ as $\T(\bv) - \bv$, we obtain the following unconstrained optimization problem with respect to the transport map
\begin{equation}\label{ism_tr}
    \inf_{\T} ~ \int_{\R^{d_v}} \left(\varepsilon |\T(\bv) - \bv|^2 + 2\Delta t \tr\left((\nabla_{\bv} \T(\bv))^{-1}\right)\right) f^n(\bv) \rd\bv \,.
\end{equation}

Now, if we use Monge's formulation to represent the Wasserstein distance, the JKO scheme for the heat equation \eqref{heat} takes the form 
\begin{align} \label{0409}
    \inf_\T ~ \int_{\R^{d_v}} \varepsilon|\T(\bv)-\bv|^2 f^n(\bv) \rd \bv + 2\Delta t \mathcal{H}(\T\sharp f^n)\,.
\end{align}
Invoking the relation \eqref{0402} and discarding terms that depend only on $f^n$, it becomes
\begin{equation}\label{heat_monge}
    \inf_{\T}  \int_{\R^{d_v}} \left( \varepsilon |\T(\bv) - \bv|^2 - 2\Delta t \log|\det \nabla_{\bv} \T(\bv)| \right) f^n(\bv) \rd\bv \,.
\end{equation}

Comparing \eqref{ism_tr} with \eqref{heat_monge}, we observe that \eqref{ism_tr} is more unstable from an optimization perspective. This is because if we assume that $\nabla_{\bv} \T(\bv)$ is invertible and has eigenvalues $\lambda_1, \cdots, \lambda_{d_v}$, then the second terms in \eqref{ism_tr} and \eqref{heat_monge} are given, respectively, by
\begin{equation*}
  \tr\left((\nabla_{\bv} \T(\bv))^{-1}\right) = \sum_{i=1}^{d_v} \frac{1}{\lambda_i} 
   \quad \text{and} \quad 
     -\log|\det \nabla_{\bv} \T(\bv)| = - \sum_{i=1}^{d_v} \log |\lambda_i|  \,.
\end{equation*}
As $\lambda_i \rightarrow 0$, the expression in \eqref{ism_tr} becomes significantly more singular near small singular values of $\nabla_{\bv} \T(\bv)$, leading to steeper gradients and less stable optimization.

More importantly, by writing down the optimality conditions, we have 
\begin{flalign}
    \T(\bv) - \bv &= -\tfrac{\Delta t}{\varepsilon} \left((\nabla \T^{-1})^\top \nabla\log f^{n+1} + \Delta \T^{-1} \right) (\T(\bv)) ~~ \text{if $\det\nabla_{\bv} \T(\bv) > 0$} \,, \label{opt_tr} \\
    \T(\bv) - \bv &= -\tfrac{\Delta t}{\varepsilon} \nabla\log f^{n+1}(\T(\bv)) \,, \label{opt_det}
\end{flalign}
corresponding to \eqref{ism_tr} and \eqref{heat_monge}, respectively. 

It is clear from \eqref{opt_det} that the JKO formulation with Monge's representation of the Wasserstein distance \eqref{0409} yields the {\it implicit score} and consequently leads to a desirable implicit particle update. In contrast, the one-step approximation of the dynamic formulation of the Wasserstein distance \eqref{ism} introduces a discrepancy from the implicit score. Indeed, when $\varepsilon = \mathcal{O}(1)$, we typically have $\T = I_{d_v} + \mathcal{O}(\Delta t)$, and thus \eqref{opt_tr} and \eqref{opt_det} agree at leading order. However, when $\varepsilon \ll 1$, the two expressions differ significantly. More details of the derivation are provided in Appendix~\ref{apdx:C}.

It is very interesting to note that \eqref{ism} is closely related to an {\it implicit score matching approach}. Indeed, starting from the score matching perspective, one seeks an approximation $\bs$ to the current score function $\nabla_\bv \log f^n$. Since we already have access to samples from $f^n$, this is the explicit score matching formulation discussed earlier in \eqref{esm}.
In comparison, implicit score matching aims to approximate the future score $\nabla_\bv \log f^{n+1}$ by seeking a function $\bs$ that matches this quantity. Formally, this amounts to solving
\begin{align} \label{0403}
    \inf_{\bs} \int_{\R^{d_v}} \left|\bs(\bv) + \tfrac{\Delta t}{\varepsilon} \nabla_\bv \log f^{n+1}(\bv) \right|^2 f^{n+1} (\bv)\rd \bv \,,
\end{align}
where $f^{n+1} = \T \sharp f^n$, and $\T(\bv) = \bv + \bs(\T(\bv))$. Since $f^{n+1}$ depends on $\bs$ in an intertwined manner, it is natural to consider a fixed-point iteration for \eqref{0403}, i.e., to construct a sequence of approximations $\{\bs_{(m)}\}_m$ by solving 
\begin{align} \label{0404}
    \inf_{\bs_{(m)}} \int_{\R^{d_v}} \left|\bs_{(m)}(\bv) + \tfrac{\Delta t}{\varepsilon} \nabla_\bv \log f^{n+1}_{(m-1)}(\bv) \right|^2 f_{(m-1)}^{n+1} (\bv)\rd \bv \,,
\end{align}
where $f_{(m-1)}^{n+1} = \T_{(m-1)} \sharp f^n$ and 
\begin{align} \label{0405}
    \T_{(m-1)}(\bv) = \bv + \bs_{(m-1)}(\T_{(m-1)}(\bv)) \,.
\end{align}
Since $\T_{(m-1)}$ is independent of $\bs_{(m)}$, \eqref{0404} simplifies to  
\begin{align} \label{0406}
    \inf_{\bs_{(m)}} \int_{\R^{d_v}} \left[|\bs_{(m)}(\bv)|^2 - 2\tfrac{\Delta t}{\varepsilon} \nabla_\bv \!\cdot\! \bs_{(m)}(\bv) \right]^2 f_{(m-1)}^{n+1} (\bv)\rd \bv \,.
\end{align}
Combining \eqref{0406} with \eqref{0405}, if the fixed-point iteration converges, that is, $\T_{(m-1)} \rightarrow \T$, $\bs_{(m-1)} \rightarrow \bs$, $\bs_{(m)} \rightarrow \bs$ in an appropriate sense, then \eqref{0406} together with \eqref{0405} can be interpreted as solving 
\begin{equation} \label{ism2}
\begin{aligned}
    & \inf_{\bs} ~ \int_{\R^{d_v}} [ |\bs(\bv)|^2 - 2\tfrac{\Delta t}{\varepsilon} \nabla_{\bv} \!\cdot\! \bs(\bv)] f^{n+1}(\bv) \rd\bv \,, \\
    & s.t. ~~ \T(\bv) = \bv + \bs(\T(\bv)) \,,
\end{aligned}
\end{equation}
which is exactly the same form as \eqref{ism}. The previous discussion has already addressed the issues with this approach. 

Perhaps a more intuitive explanation can be seen by comparing \eqref{ism2} with \eqref{0403}. A key component in score matching is the omission of the nonlinear term $\int_{\R^{d_v}} |\nabla_{\bv} \log f|^2 f \rd\bv$ in the loss function, which is justified when $f$ is independent of the score function being learned. However, in the implicit formulation \eqref{0403}, this term should not be neglected, since $f^{n+1}$ depends on the unknown score function in a nonlinear and coupled manner. The fixed-point iteration effectively drops this term, which can lead to inconsistencies and, ultimately, problematic behavior.

\begin{remark}
    We would like to clarify that our criticism of implicit score matching \eqref{0403} pertains to its implementation via fixed-point iteration \eqref{0404}; we do not claim that \eqref{0403} is incorrect. On the contrary, it is theoretically sound, but challenging to implement in practice.
    One might further ask: since \eqref{0409} provides the correct implicit update (i.e.,  \eqref{opt_det}), why not use it directly? While this is feasible for the heat equation, the key issue is that Monge’s formulation does not readily generalize to the Landau metric or to projected Wasserstein gradient flows for Dougherty collision. In contrast, our dynamic formulation offers significantly greater flexibility in these settings.
\end{remark}

\begin{remark}
    If one instead uses a forward Euler discretization together with the right-endpoint quadrature rule, a similar issue to that in \eqref{ism} arises. In this case, \eqref{heat_jko_div} becomes 
\begin{equation*}
\begin{aligned}
    & \inf_{\bs} ~  \int_{\R^{d_v}} \left[\varepsilon |\bs(\T(\bv))|^2 - 2 \Delta t (\nabla \!\cdot\! \bs) ( \T(\bv)) \right] f^n(\bv) \rd\bv \,, \\
    & s.t. ~~ \T(\bv) = \bv + \bs(\bv) \,,
\end{aligned}
\end{equation*}
which is equivalent to the unconstrained formulation
\begin{align} \label{0401}
    \inf_{\bs} ~  \int_{\R^{d_v}} \left[\varepsilon |\bs(\bv + \bs (\bv))|^2 - 2 \Delta t (\nabla \!\cdot\! \bs)( \bv + \bs(\bv)) \right] f^n(\bv) \rd\bv \,.
\end{align}
The optimality condition of \eqref{0401} is given by
\begin{equation*}
    \bs = -\tfrac{\Delta t}{\varepsilon} \nabla\log f^{n+1} + \tfrac{f^n}{f^{n+1}} \left[\tfrac{\Delta t}{\varepsilon} (\nabla(\nabla \cdot \bs))\circ \T - (\nabla\bs) \circ \T \right] \,,
\end{equation*}
where $f^{n+1} = \T \sharp f^n$. Therefore, when $\varepsilon \ll \Delta t$, $\bs$ deviates significantly from the implicit score $-\frac{\Delta t}{\varepsilon} \nabla\log f^{n+1}$, again leading to an incorrect update in the stiff regime. Derivations are given in Appendix~\ref{apdx:C}. 
\end{remark}

\paragraph{Multi-step inner discretization is necessary}
All in all, the above discussion suggests that a one-step inner discretization is insufficient for stiff problems, and that a more accurate evaluation of the log-determinant term is necessary. It also highlights that our dynamical JKO formulation accommodates a much broader class of discretization strategies, providing a principled and flexible framework for handling stiffness. 

In practice, we discretize the inner-time integral terms in \eqref{heat_jko_div} using a Gauss--Legendre quadrature rule. Specifically, let $\{\tau_k\}_{k=1}^K$ be the Gauss--Legendre nodes on $[0,1]$ and set $\tau_0=0$ and $\tau_{K+1}=1$. These points partition the inner-time interval $[0,1]$ into $K+1$ subintervals $[\tau_{k}, \tau_{k+1}]$, $k=0, \ldots, K$. On each subinterval, we solve the ODE constraint in \eqref{heat_jko_div} using a high-order ODE solver to obtain $\bz^k(\bv)\approx \T_{\tau_k}(\bv)$. We then approximate the loss by
\begin{equation}\label{GL}
    \sum_{k=1}^{K} q_k \int_{\R^{d_v}}
    \left[
    \varepsilon \, |\bs(\tau_k,\bz^k(\bv))|^2
    - 2\Delta t \, (\nabla\!\cdot\!\bs)(\tau_k,\bz^k(\bv))
    \right] f^n(\bv)\,\rd\bv,
\end{equation}
where $q_k$ denotes the quadrature weight associated with the node $\tau_k$.

It is also important to relate our problem \eqref{heat_jko_div} to \textit{continuous flow-based generative models}. In fact, when \(\varepsilon=0\), it reduces exactly to continuous normalizing flows (CNFs) \cite{neuralode, grathwohl2018scalable}. For \(0<\varepsilon\ll \Delta t\), \eqref{heat_jko_div} is closely connected to OT-Flow \cite{pmlr-v119-finlay20a, Onken_Wu_Fung_Li_Ruthotto_2021}. In these works, it has also been recognized that accurate evaluation of the log-determinant term is crucial in practice. A key distinction in our formulation, compared to these prior works on generative models, is that instead of solving the auxiliary ODE in \eqref{heat_jko_div}, we rewrite it in terms of inner-time integrals, since the associated velocity fields do not explicitly depend on \(\log|\det\nabla_{\bv}\T_\tau|\). As a result, \(\log|\det\nabla_{\bv}\T_1|\) can be computed directly via quadrature using the nodes generated by the ODE solver, without introducing additional auxiliary variables. This leads to improved efficiency, and quadrature is typically more accurate than direct ODE propagation for this task. Our previous discussion also provides a clearer justification for the necessity of multistep discretization and integration.

\subsection{Extension to the full equation}\label{sec:2.3}
One challenge in extending the space-independent collision step to the spatially dependent setting is that collisions are inherently local in $\bx$. Consequently, a mechanism is needed to identify neighboring particles. A natural approach is to use spatial binning: partition the domain into cells and assign each particle to a cell. Particles within the same cell are then allowed to interact (collide). This leads naturally to an operator-splitting scheme for particle updates, separating the free transport step from the collision step \cite{Neunzert_Struckmeier_1995, MC_Boltz}. 

Here, we consider a first-order splitting scheme; higher-order extensions should be straightforward. The free-transport step $\partial_t f + \bv \cdot \nabla_{\bx} f = 0$ is advanced exactly:
\begin{equation}\label{transport_step}
    f^*(\bx, \bv) = f^n(\bx - \Delta t\bv, \bv) \,.
\end{equation}
We then partition the $\bx-$domain into a collection of spatial cells $\Omega = \cup_{l=1}^{N_c}\Omega_l$ and define the cell average
\begin{equation}\label{cell_avg}
    f_l^*(\bv) = \frac{1}{|\Omega_l|} \int_{\Omega_l} f^*(\bx, \bv) \,\rd\bx ~~\text{if}~ x \in \Omega_l \,.
\end{equation}
In each cell $\Omega_l$, we solve the following collision problem for one time step 
\begin{equation*}
    \partial_t f = \frac{1}{\varepsilon} \Qop(f) \,,~~ f(0, \bv) = f_l^*(\bv) \,, 
\end{equation*}
and obtain $f_l^{n+1}(\bv):= f(\Delta t, \bv)$. This step is treated implicitly through the variational scheme introduced above. Finally,  one can reconstruct $f^{n+1}(\bx, \bv)$ from the cell-average $\{f_l^{n+1}(\bv)\}_{l=1}^{N_c}$. Further details will be provided in the next section when we present the fully discrete scheme.

We also note that there is an alternative way to identify neighboring particles, based on a delocalization in $\bx$ via convolution with an auxiliary kernel. For instance, in the case of the Landau collision operator, one may consider a delocalized approximation \cite{bailo2024collisional}: 
\begin{equation*}
    \Qop_L^{\kappa}(f, f) = \nabla_{\bv} \cdot \left( f\int_{\Omega \times \R^{d_v}} \kappa(\bx-\bx_*) A(\bv-\bv_*) (\nabla_{\bv}\log f - \nabla_{\bv_*} \log f_*) f_* \rd\bv_* \rd\bx_* \right) ,
\end{equation*}
where $f := f(\bx, \bv)$ and $f_* := f(\bx_*, \bv_*)$. When the kernel $\kappa$ is a Dirac delta function, $\mathcal Q_L^\kappa$ reduces to $\mathcal Q_L$. For smooth kernel $\kappa$, $\mathcal Q_L^\kappa$ is often referred to as the fuzzy Landau operator \cite{gualdani2025fuzzylandau}.

With $\mathcal{Q}_L^\kappa$, we extend the dynamic JKO formulation for the homogeneous Landau equation to the spatially dependent setting \eqref{eqn0} as follows. Given $f^n(\bx,\bv)$, we solve $f(\tau, \bx, \bv)$ by:
\begin{equation*}
\begin{split}
    & \inf_{f,\bs} ~ \varepsilon \int_0^1 \frac12 \int_{(\Omega \times \R^{d_v})^2} \!\! \kappa(\bx - \bx_*) |\bs - \bs_*|^2_{A(\bv - \bv_*)} ff_* \rd\bv_* \rd\bv \rd\bx_* \rd\bx \rd\tau + 2\Delta t \mathcal{S}(f(1, \cdot, \cdot)) \,, \\
    & s.t. ~~ \partial_\tau f + \Delta t \bv \!\cdot\! \nabla_{\bx}f + \nabla_{\bv} \!\cdot\! \left[f \! \int_{\Omega \times \R^{d_v}} \!\! \kappa(\bx-\bx_*) A(\bv-\bv_*)(\bs-\bs_*) f_* \rd\bv_* \rd\bx_* \right] = 0 \,, \\
    & ~~~~~~ f(0, \cdot, \cdot) = f^n \,,
\end{split}
\end{equation*}
where $\bs := \bs(\tau, \bx, \bv)$, $\bs_* := \bs(\tau, \bx_*, \bv_*)$ and $\mathcal{S}(f) = \int_{\Omega \times \R^{d_v}} f \log f \rd\bv\rd\bx$ is the global entropy. We then set $f^{n+1}(\bx, \bv):= f(1, \bx, \bv)$. 

Compared with the operator-splitting approach described above, the main advantage of this delocalized formulation is that it avoids splitting errors and provides a global entropy dissipation mechanism. However, in the presence of spatial discontinuities, such as the Riemann problem, it may be difficult to parametrize $\bs(\tau,\bx,\bv)$ globally using the neural networks considered later, since such models often struggle to approximate discontinuous profiles. In contrast, operator splitting decouples the spatial dependence and reduces the collision step to a collection of independent homogeneous problems within each cell, making the representation of $\bs(\tau,\bv)$ much simpler. For this reason, we adopt the operator-splitting strategy in the rest of this paper.

\section{Numerical schemes}\label{sec:3}
In this section, we present the fully discrete numerical scheme along with implementation details. This includes a particle-based spatial discretization, neural network approximations, and training procedures. We also highlight the structure-preserving properties of the proposed method.

\subsection{Particle discretization and neural network}
Let $\{\bx_i^0\,, \bv_i^0 \}_{i=1}^N$ denote $N$ particles sampled from the initial data $f^0$, each assigned a weight $w = \frac{m}{N}$, where $m:=\int_{\Omega \times \R^{d_v}} f^0 \rd\bv\rd\bx$ is the total mass. Let $\{\Omega_l\}_{l=1}^{N_c}$ be a fixed partition of the spatial domain $\Omega$. In what follows, we present the particle update from time $t_n = n\Delta t$ to $t_{n+1} = (n+1)\Delta t$ for $n=0,1,2,\cdots$, in an operator splitting framework as discussed in Section~\ref{sec:2.3}. 

\paragraph{Transport step}
The particle interpretation of the transport step \eqref{transport_step} is
\begin{equation*}
    \bx^*_i = \bx_i^n + \Delta t \bv_i^n \,,~ i=1,\cdots,N \,.
\end{equation*}
To implement periodic and reflecting boundary conditions in space, we follow the approach described in \cite{MC_Boltz}. For clarity, we present them in the setting of a one-dimensional spatial domain $\Omega_x = [L_{-}, L_{+}]$.
\begin{equation*}
\text{Periodic B.C.} \quad
(\bx_i^* \,, \bv^*_i) = \left \{
\begin{aligned}
    & (\bx^*_i \,, \bv_i^n) \,,~~ \text{if}~~ L_{-} \leq \bx^*_i \leq L_{+} \,, \\
    & (\bx^*_i + (L_{+} - L_{-}) \,, \bv_i^n) \,,~~ \text{if}~~ \bx^*_i < L_{-} \,, \\
    & (\bx^*_i - (L_{+} - L_{-}) \,, \bv_i^n) \,,~~ \text{if}~~ \bx^*_i > L_{+} \,.
\end{aligned}
\right.
\end{equation*}
\begin{equation*}
\text{Reflecting B.C.} \quad
(\bx_i^* \,, \bv^*_i) = \left \{
\begin{aligned}
    & (\bx^*_i \,, \bv_i^n) \,,~~ \text{if}~~ L_{-} \leq \bx^*_i \leq L_{+} \,, \\
    & (L_{-} - (\bx^*_i - L_{-}) \,, -\bv_i^n) \,,~~ \text{if}~~ \bx^*_i < L_{-} \,, \\
    & (L_{+} - (\bx^*_i - L_{+}) \,, -\bv_i^n) \,,~~ \text{if}~~ \bx^*_i > L_{+} \,.
\end{aligned}
\right.
\end{equation*}

The density remains unchanged, since both periodic and reflecting boundary conditions are volume-preserving. In other words,
\begin{equation*}
    f^*(\bx_i^*, \bv_i^*) = f^n(\bx_i^n, \bv_i^n) \,,~ i=1,\cdots,N \,.
\end{equation*}

\paragraph{Binning}
We assign all particles $\{(\bx_i^* \,, \bv^*_i)\}_{i=1}^N$ to the cells $\{\Omega_l\}_{l=1}^{N_c}$ according to their spatial location. We then interpret the set $\{\bv_i^*\}_{i \in \Omega_l}$ as particles sampled from the approximated distribution $f_l^*(\bv)$ \eqref{cell_avg}, where $i \in \Omega_l$ represents the indices of particles whose spatial positions satisfy $\bx_i^* \in \Omega_l$.

\paragraph{{Collision step}}
Since the collision step is applied independently within each cell, we describe the discretization for a representative cell $\Omega_l$. 

To conserve the total mass, due to the definition of the local density in \eqref{cell_avg}, we reweight the particles according to their assigned cells. More precisely, let $\tilde{w}$ denote the particle weight and $N_l$ the number of particles in $\Omega_l$. Then it should satisfy
\begin{equation*}
    \underbrace{\int_{\R^{d_v}} \tilde{w} \sum_{i \in \Omega_l} \delta(\bv-\bv_i^*) \rd\bv}_{\int_{\R^{d_v}} f_l^*(\bv)\rd\bv} = \underbrace{\frac{1}{|\Omega_l|} \int_{\Omega_l \times \R^{d_v}} \!\! w \sum_{i = 1}^N \delta(\bx-\bx_i^*) \delta(\bv-\bv_i^*) \rd\bx \rd\bv}_{\frac{1}{|\Omega_l|}\int_{\Omega \times \R^{d_v}} f^*(\bx, \bv)\rd \bv \rd\bx}  \,,
\end{equation*}
The left-hand side equals $\tilde{w} N_l$, while the right-hand side equals $\frac{w N_l}{|\Omega_l|}$, which yields $\tilde{w} = \frac{w}{|\Omega_l|}$.

The resulting semi-discrete variational schemes are: \\
\noindent\textit{Landau}:
\begin{equation}\label{LD_PJKO}
\begin{aligned}
    & \inf_{\bs}~ \tilde{w}^2 \!\! \sum_{i, j \in \Omega_l} \varepsilon \int_0^1 \frac12 |\bs(\tau, \T_\tau(\bv_i^*)) - \bs(\tau, \T_\tau(\bv_j^*))|^2_{A(\T_\tau(\bv_i^*) - \T_\tau(\bv_j^*))} \rd\tau -2 \Delta t \tilde{w} \sum_{i \in \Omega_l} \ell_i  \,, \\
    & \text{s.t.}~ \tfrac{\rd}{\rd\tau} \T_\tau(\bv_i^*) = \tilde{w} \sum_{j \in \Omega_l} A(\T_\tau(\bv_i^*) \!-\! \T_\tau(\bv_j^*)) [\bs(\tau, \T_\tau(\bv_i^*)) \!-\! \bs(\tau, \T_\tau(\bv_j^*))] \,, \T_0(\bv_i^*)=\bv_i^* \,, \\
    & \qquad \ell_i = \tilde{w}\displaystyle\sum_{j \in \Omega_l} \int_0^1 \nabla_{\bsv} \!\cdot\! [A(\bsv - \T_\tau(\bv_j^*))(\bs(\tau, \bsv) - \bs(\tau, \T_\tau(\bv_j^*)))] \big|_{\bsv = \T_\tau(\bv_i^*)} \,\rd\tau \,.
\end{aligned}
\end{equation}

\noindent\textit{Dougherty}:
\begin{equation}\label{Dou_PJKO}
\begin{split}
    &\inf_{\bs}~ \tilde{w} \sum_{i \in \Omega_l} \varepsilon \int_0^1 |\bs^\perp(\tau, \T_\tau(\bv_i^*))|^2 \rd\tau - 2 \Delta t T_l^* \tilde{w} \sum_{i \in \Omega_l} \ell_i  \,, \\
    &\text{s.t.}~ \tfrac{\rd}{\rd\tau} \T_\tau(\bv_i^*) = \bs^{\perp}(\tau, \T_\tau(\bv_i^*)) \,,~ \T_0(\bv_i^*)=\bv_i^* \,, \\
    & \qquad \ell_i = \int_0^1 \nabla_{\bsv} \!\cdot\! \bs^\perp(\tau, \bsv) \big|_{\bsv = \T_\tau(\bv_i^*)} \,\rd\tau \,.
\end{split}
\end{equation}
Here, the particle discretized version of $\bs^{\perp}$ \eqref{s_perp} is given by
\begin{flalign*}
    \bs^{\perp}(\tau, \T_\tau(\bv_i^*)) = \bs(\tau, &\T_\tau(\bv_i^*)) - \frac{1}{N_l}\sum_{j \in \Omega_l} \bs(\tau, \T_\tau(\bv_j^*)) \\
    & - \frac{\sum_{j \in \Omega_l}  \bs(\tau, \T_\tau(\bv_j^*)) (\T_\tau(\bv_j^*) - \bu_{l}(\tau))}{\sum_{j \in \Omega_l} |\T_\tau(\bv_j^*) - \bu_{l}(\tau)|^2 } (\T_\tau(\bv_i^*) - \bu_{l}(\tau)) \,,
\end{flalign*}
where $\bu_{l}(\tau) = \frac{1}{N_l} \sum_{i \in \Omega_l} \T_\tau(\bv_i^*)$.

Upon optimization, particles and densities are updated according to
\begin{flalign*}
    & \bx_i^{n+1} = \bx_i^* \,,~~ \bv_i^{n+1} = \T_1(\bv_i^*) \,, \\
    & f^{n+1}(\bx_i^{n+1}, \bv_i^{n+1}) = f^*(\bx_i^*, \bv_i^*) / \exp(\ell_i) \,.
\end{flalign*}

In practice, to approximate $\bs(\tau, \bv)$, we employ a multi-layer perceptron (MLP) $\bs_\theta(\tau, \bv): \R^{d_v+1} \to \R^{d_v}$ with parameters $\theta$. Specifically, an $L$-layer MLP with $m$ neurons per hidden layer, taking input $\bar{\bv} = (\tau, \bv)$, is defined as
\begin{equation}\label{MLP}
\bs_\theta(\bar{\bv}) = W_L \sigma\Big( W_{L-1} \big( \cdots \sigma( W_2 \sigma( W_1 \bar{\bv} + b_1 ) + b_2 ) \cdots \big) + b_{L-1} \Big) + b_L \,,
\end{equation}
where $\theta = \{W_l, b_l\}_{l=1}^L$ are trainable, time-independent parameters. The weight matrices and biases satisfy: $W_1 \in \R^{m \times (d_v+1)}$, $\{W_l\}_{l=2}^{L-1} \in \R^{m \times m}$, $W_L \in \R^{d_v \times m}$, $\{b_l\}_{l=1}^{L-1} \in \R^{m}$, $b_L \in \R^{d_v}$. The function $\sigma$ is an element-wise activation function. Compared with traditional representations such as polynomial or Fourier bases, neural networks offer notable advantages: they scale effectively with high-dimensional inputs and can naturally incorporate the dependence on the inner-time variable $\tau$ through input.

To conclude this section, we summarize the structural properties of the proposed approach. 
For each cell $\Omega_l$, we define the discrete local density, velocity, temperature, and entropy as:
\begin{equation*}
    (\rho_l^n, \bu_l^n, T_l^n, H_l^n) = \left(\tilde{w} N_l , \frac{1}{N_l} \sum_{i \in \Omega_l} \bv_i^n , \frac{1}{d_v N_l} \sum_{i \in \Omega_l} |\bv_i^n-\bu_l^n|^2 , \tilde{w} \sum_{i \in \Omega_l} \log f^n(\bx_i^n, \bv_i^n) \right) .
\end{equation*}
Correspondingly, the global quantities are defined by:
\begin{equation*}
    (\rho_{tot}^n, \bu_{tot}^n, E_{tot}^n, H_{tot}^n) = \left(\sum_{i=1}^N w ~, w \sum_{i=1}^N \bv_i^n ~, w\sum_{i=1}^N |\bv_i^n|^2 ~, w \sum_{i=1}^N \log f^n(\bx_i^n, \bv_i^n) \right) .
\end{equation*}

\begin{proposition}
    The proposed schemes \eqref{LD_PJKO} and \eqref{Dou_PJKO} satisfy:
    \begin{enumerate}[(i)]
        \item Local (collision step). $(\rho_l^{n+1}, \bu_l^{n+1}, T_l^{n+1}) = (\rho_l^*, \bu_l^*, T_l^*)$ and $H_l^{n+1}\le H_l^*$ for each cell $\Omega_l$.
        \item Global (periodic BC). $(\rho_{tot}^{n+1},\bu_{tot}^{n+1},E_{tot}^{n+1}) =(\rho_{tot}^{n},\bu_{tot}^{n},E_{tot}^{n})$ and $H_{tot}^{n+1}\le H_{tot}^n$.
    \end{enumerate}
\end{proposition}

\begin{proof}
Local claims for Landau collision follow from the same argument as \cite[Proposition 3.1]{huang2024jkolandau}. The Dougherty collision can be treated similarly, and we omit the proof for brevity. For global entropy, transport gives $H_{tot}^*=H_{tot}^n$, and
\begin{equation*}
    H_{tot}^n = w \sum_{i=1}^N \log f^n(\bx_i^n, \bv_i^n) = \sum_{l=1}^{N_c} |\Omega_l| H_l^n \,.
\end{equation*}
Thus $H_l^{n+1}\le H_l^*$ for all $l$ implies
$H_{tot}^{n+1}\le H_{tot}^*=H_{tot}^n$.
Global mass, momentum, and energy conservation follow from local conservation in collision plus conservation in transport (periodic BC).
\end{proof}

\subsection{Inner-time discretization}
The partition of the inner-time interval $\tau\in [0,1]$ into subintervals, along with the associated numerical quadrature, follows the strategy in \eqref{GL}. We now describe in detail how to update the ODE systems in \eqref{LD_PJKO} and \eqref{Dou_PJKO} on each subinterval.

Our default choice is the explicit fourth-order Runge–Kutta method (RK4), which provides a good balance between computational efficiency and accuracy. A potential drawback is that it does not exactly conserve energy at the discrete level. Nevertheless, numerical experiments indicate that the resulting energy error remains sufficiently small; see Table~\ref{tab:homo_err} in the next section.

If exact energy conservation at the discrete level is desired, we instead employ the implicit midpoint method (imRK2) \cite{Hirvijoki_2021, Jeyakumar_Kraus_Hole_Pfefferlé_2024}. More specifically, let $\bz_i^k \approx \T_{\tau^k}(\bv_i^*)$ in \eqref{LD_PJKO} and \eqref{Dou_PJKO} for $k = 0, \cdots K$. Then 
\begin{equation}\label{imRK2}
    \begin{split}
        \bz_i^{k+1}
        &= \bz_i^k + \Delta\tau \tilde w\sum_{j\in\Omega_l}
        A\!\left(\bz_i^{k+\frac12} \!-\! \bz_j^{k+\frac12}\right)
        \Big[\bs_\theta\!\left(\tau_{k+\frac12},\bz_i^{k+\frac12}\right)
            \!-\! \bs_\theta\!\left(\tau_{k+\frac12},\bz_j^{k+\frac12}\right)\Big], \\
        \bz_i^{k+1}
        &= \bz_i^k + \Delta\tau\, \bs_\theta^{\perp}\!\left(\tau_{k+\frac12},\bz_i^{k+\frac12}\right), 
    \end{split}
\end{equation}
for Landau and Dougherty collisions, respectively. 
Here \(\Delta\tau=\tau_{k+1}-\tau_k\), \(\tau_{k+\frac12}=\frac{\tau_k+\tau_{k+1}}{2}\), and \(\bz_i^{k+\frac12}=\frac{\bz_i^k+\bz_i^{k+1}}{2}\).

In practice, \eqref{imRK2} requires solving a nonlinear system at each inner-time step, and the system is increasingly ill-conditioned as $\varepsilon \to 0$. To resolve this challenge, we employ Broyden's method, a quasi-Newton scheme that mitigates ill-conditioning while avoiding costly Jacobian evaluations through inexpensive rank-one updates.

To illustrate the idea, we reformulate \eqref{imRK2} as a fixed-point system $\mathcal{G}(\mathbf{z})=\mathbf{z}$, where $\mathbf{z}:=(\bz_i^{k+1})_{i\in\Omega_l}$ collects all particle velocities in the spatial cell $\Omega_l$. Broyden's method iterates through
\begin{equation*}
\begin{aligned}
    &\mathbf{z}_{(m+1)}=\mathbf{z}_{(m)}-\eta_m J_{(m)}^{-1} (\mathcal{G}(\mathbf{z}_{(m)}) - \mathbf{z}_{(m)}) \,, \\
    &\Delta\mathbf{z}_{(m)}=\mathbf{z}_{(m+1)}-\mathbf{z}_{(m)} \,, \quad
    \Delta\mathcal{G}_{(m)}=\mathcal{G}(\mathbf{z}_{(m+1)}) - \mathbf{z}_{(m+1)} - \mathcal{G}(\mathbf{z}_{(m)}) + \mathbf{z}_{(m)} \,, \\
    &J_{(m+1)}^{-1} = J_{(m)}^{-1} + \frac{\Delta\mathbf{z}_{(m)} - J_{(m)}^{-1} \Delta\mathcal{G}_{(m)}}{(\Delta\mathbf{z}_{(m)})^\top J_{(m)}^{-1} \Delta\mathcal{G}_{(m)}} (\Delta\mathbf{z}_{(m)})^\top J_{(m)}^{-1} \,.
\end{aligned}
\end{equation*}
Here, $J_{(m)}^{-1}$ is the approximated inverse Jacobian of $\mathcal{G}_{(m)}$. The initial iterate $\mathbf{z}_{(0)}$ is provided by an explicit predictor (e.g., RK4 from $(\bz_i^k)_{i \in \Omega_l}$), and $J_{(0)}^{-1} = I$. 

To safeguard convergence, we equip with an Armijo backtracking line search on
the merit function $\Phi(\mathbf{z})=\|\mathcal{G}(\mathbf{z}) - \mathbf{z}\|_2^2$. The step size is set to $\eta_m=\beta^{r_m}$ with $0<\beta<1$, where $r_m$ is the smallest non-negative integer satisfying
\begin{equation*}
    \Phi(\mathbf{z}_{(m)}-\eta_m J_{(m)}^{-1}\mathcal{G}(\mathbf{z}_{(m)}))
    \le (1 - 2c \eta_m)\Phi(\mathbf{z}_{(m)}) \,,
\end{equation*}
with $c\in(0,1)$. The iteration terminates when $\|\mathcal{G}(\mathbf{z}_{(m)}) - \mathbf{z}_{(m)}\|_2 \le \texttt{tol}$ or upon reaching the maximum iteration count.

\subsection{Efficient training strategies}
The main practical bottleneck of implicit schemes lies in backpropagation. Training requires the gradient $\frac{\rd \mathbf{z}}{\rd \theta}$, while $\mathbf{z}$ is the fixed point produced by an iterative solver. Naively unrolling all iterations and applying automatic differentiation is both memory intensive and numerically fragile. To avoid this difficulty, one may instead compute the gradient via implicit differentiation. Observe that the dependence of \eqref{imRK2} on the trainable parameters $\theta$ is characterized by $\mathbf{z}(\theta) = \mathcal{G}(\mathbf{z}(\theta),\theta)$. Hence
\begin{equation*}
    \frac{\rd\mathbf{z}}{\rd\theta} = \left(I - \frac{\partial\mathcal{G}}{\partial\mathbf{z}} \right)^{-1} \frac{\partial \mathcal{G}}{\partial\theta}\,.
\end{equation*}
This, however, involves solving a Jacobian ($\frac{\partial\mathcal{G}}{\partial\mathbf{z}}$) dependent linear systems, which can be costly and cumbersome in practice. To address this issue, we adopt \textit{Jacobian-free backpropagation} (JFB) \cite{Fung_Heaton_Li_Mckenzie_Osher_Yin_2022}, which avoids such linear solves in the backward pass by simply setting
\begin{equation*}
    \frac{\rd\mathbf{z}}{\rd\theta} = \frac{\partial\mathcal{G}}{\partial\theta} \,.
\end{equation*}
Intuitively, this amounts to rewriting the fixed point system as $\mathbf{z}(\theta) = \mathcal{G}(\bar{\mathbf{z}}, \theta)$, where $\bar{\mathbf{z}} = \mathbf{z}$ denotes a stop-gradient copy of $\mathbf{z}$, detached from the dependence on $\theta$. Alternatively, this can be interpreted as adopting a zeroth-order approximation of the Neumann series  $\left(I - \frac{\partial\mathcal{G}}{\partial\mathbf{z}} \right)^{-1} = \sum_{k=0}^{\infty} (\frac{\partial\mathcal{G}}{\partial\mathbf{z}})^k$. The validity of this treatment has been analyzed in both the static case \cite{Fung_Heaton_Li_Mckenzie_Osher_Yin_2022} and the more relevant time-dependent case \cite{gelphman2026}.

Such modification provides fixed-memory training, a simpler implementation, and substantially improved computational efficiency. Specifically, we summarize the inner-time discretizations of \eqref{LD_PJKO} and \eqref{Dou_PJKO}, together with the implicit scheme \eqref{imRK2}, in the following abstract form:
\begin{flalign*}
    &\inf_{\theta}~ \mathcal{L} \!\left(\{\bz^k(\theta)\}_{k=0}^{K+1}, \theta \right) \,, \\
    &\qquad \text{s.t.}~~ \mathcal{G}(\bz^{k+1}(\theta), \bz^{k}(\theta), \theta) = \bz^{k+1}(\theta) ~~ \text{for}~~ k=0,1,\cdots,K \,.
\end{flalign*}
The implementation of JFB in the PyTorch library is summarized in Algorithm~\ref{alg:JFB}.

\begin{algorithm}[ht!]
\caption{Jacobian-free backpropagation}
\label{alg:JFB}
\begin{algorithmic}[1]
    \FOR{$k=0,1,\cdots,K$}
        \STATE \texttt{with torch.no\_grad():}
        \STATE \qquad Solve the fixed-point system $\mathcal{G}(\mathbf{z}, \bz^{k}(\theta),\theta)=\mathbf{z}$
        by an iterative solver.
        \STATE Re-evaluate $\bz^{k+1}(\theta)=\mathcal{G}(\mathbf{z}, \bz^{k}(\theta),\theta)$ to construct the computation graph.
    \ENDFOR
    \STATE Compute the $\operatorname{loss}=\mathcal{L} \!\left(\{\bz^k(\theta)\}_{k=0}^{K+1},\theta \right)$.
    \STATE \texttt{loss.backward()}
    \STATE \texttt{optimizer.step()}
\end{algorithmic}
\end{algorithm}

Another challenge arises in the computation of the Landau problem \eqref{LD_PJKO}, where the pairwise interaction terms dominate the computational cost, yielding a complexity of $\mathcal{O}(N_l^2)$ per cell $\Omega_l$. Following our previous work \cite{huang2024jkolandau}, we employ the same stochastic acceleration strategy by combining mini-batch SGD \cite{Bottou2018} with the random batch method (RBM) \cite{JIN2020108877, Carrillo_Jin_Tang_2022}. At each epoch, particles are randomly partitioned into batches of size $B_l \ll N_l$; interactions are evaluated only within each batch, and optimizer updates are performed batch by batch. Therefore, the total cost is reduced to $\mathcal{O}(B_l N_l)$. The procedure is summarized in Algorithm \ref{alg:SGD}.

\begin{algorithm}[ht!]
\caption{Mini-batch SGD combined with RBM}
\label{alg:SGD}
\begin{algorithmic}[1]
\REQUIRE Mini-batch size $B_l$
\FOR{each training epoch}
    \STATE Randomly divide the index set $\{i : i \in \Omega_l\}$ into $\lfloor \frac{N_l}{B_l} \rfloor$ batches.
    \FOR{each batch}
        \STATE Compute the trajectory ODE and the loss within the batch.
        \STATE Perform one optimization step.
    \ENDFOR
\ENDFOR
\end{algorithmic}
\end{algorithm}

\section{Numerical examples}\label{sec:4}
This section presents numerical tests for the proposed schemes, with focus on asymptotic-preserving behavior and structure-preservation. We first consider spatially homogeneous Landau and Dougherty collisions over a range of Knudsen numbers $\varepsilon$ to assess relaxation and conservation properties. We then study the spatially inhomogeneous Landau equation, where we examine global conservation, entropy dissipation, and the asymptotic-preserving property. All Landau tests use $\gamma=-3$, and all inhomogeneous examples are one-dimensional in space.

Unless otherwise specified, we use the following default setup. Particle positions are generated by stratified sampling, and particle velocities are sampled exactly when the initial velocity marginal is Gaussian or a Gaussian mixture. The velocity field $\bs_\theta(t,\bv)$ is represented by a 5-layer MLP \eqref{MLP} with 32 neurons per hidden layer and \texttt{SiLU} activation, $\sigma(z)=\frac{z}{1+e^{-z}}$. Biases are initialized to zero, and weights are sampled from a truncated normal distribution with variance $1/\texttt{fan\_in}$. For optimization, we use \texttt{AdamW} \cite{loshchilov2018decoupled} with default parameters and \texttt{CosineAnnealingWarmRestarts} \cite{loshchilov2017sgdr} for dynamic learning-rate adjustment. The important optimization hyperparameters include
\begin{enumerate}[(1)]
    \item $\eta_{\max}$: the initial learning rate
    \item $\eta_{\min}$: the minimum learning rate.
    \item $T_0$: number of iterations until the first restart.
    \item $T_{\max}$: number of total iterations.
    \item $B$: batch size.
\end{enumerate}

To further accelerate the variational collision step, we exploit its inherent parallelizability across spatial cells. The computational workload is evenly distributed across multiple GPUs according to the cell partitions. All experiments are implemented in PyTorch with single-precision and executed on four NVIDIA L40S GPUs at the Minnesota Supercomputer Institute.

\subsection{Homogeneous Landau and Dougherty collisions}

\subsubsection{Relaxation of a bi-Maxwellian (0D-2V)}
Consider a 2-dimensional example with bi-Maxwellian initial data
\begin{equation*}
    f^0(\bv) = \frac{1}{2\pi} \left(e^{-|\bv_x - 1|^2} + e^{-|\bv_x + 1|^2} \right) e^{-\bv_y^2} \,.
\end{equation*}
We set the Knudsen number $\varepsilon = \{1, 10^{-2}, 10^{-4}, 10^{-8} \}$, the time step size $\Delta t=0.01$, and the total number of particles $N=12800$. The optimization hyperparameters are set as follows: $(\eta_{\max}, \eta_{\min}, T_0, T_{\max}, B) = (10^{-2}, 10^{-3}, 20, 100, 1280)$ for the Landau collision; $(\eta_{\max}, \eta_{\min}, T_0, T_{\max}) = (10^{-2}, 10^{-4}, 50, 200)$ for the Dougherty collision.

Figures~\ref{fig:homoLD_f} and \ref{fig:homoDougherty_f} show the one-step density $f^{\Delta t}$ for Landau and Dougherty collisions at different $\varepsilon$. The method remains stable across all tested Knudsen numbers and drives the solution toward equilibrium in the stiff regime $\varepsilon \ll \Delta t$. In particular, the bottom-right panels indicate that the distance from $f^{\Delta t}$ to the Maxwellian $\mathcal{M}$ decreases as $\varepsilon \to 0$, where 
$\|f^{\Delta t} - \mathcal{M}\|_1 := \frac{1}{N} \sum_{i=1}^N \big| f^{\Delta t}(\bv_i) - \mathcal{M}(\bv_i) \big|.$
The slight non-monotonicity at extremely small $\varepsilon$ is due to particle-resolution limits. We also observe faster relaxation for Dougherty than for Landau, consistent with \cite{Pezzi_Valentini_Veltri_2015}.

Table~\ref{tab:homo_err} reports conservation errors. For Landau, momentum is conserved uniformly in $\varepsilon$, and the corresponding RK4 energy errors remain small. For Dougherty, the projected-WGF formulation yields markedly smaller momentum and energy errors than WGF due to the explicit conservation projection. The implicit midpoint solver is energy-conserving up to nonlinear-solver tolerance.

\begin{figure}[ht]
    \centerline{\includegraphics[scale=0.33]{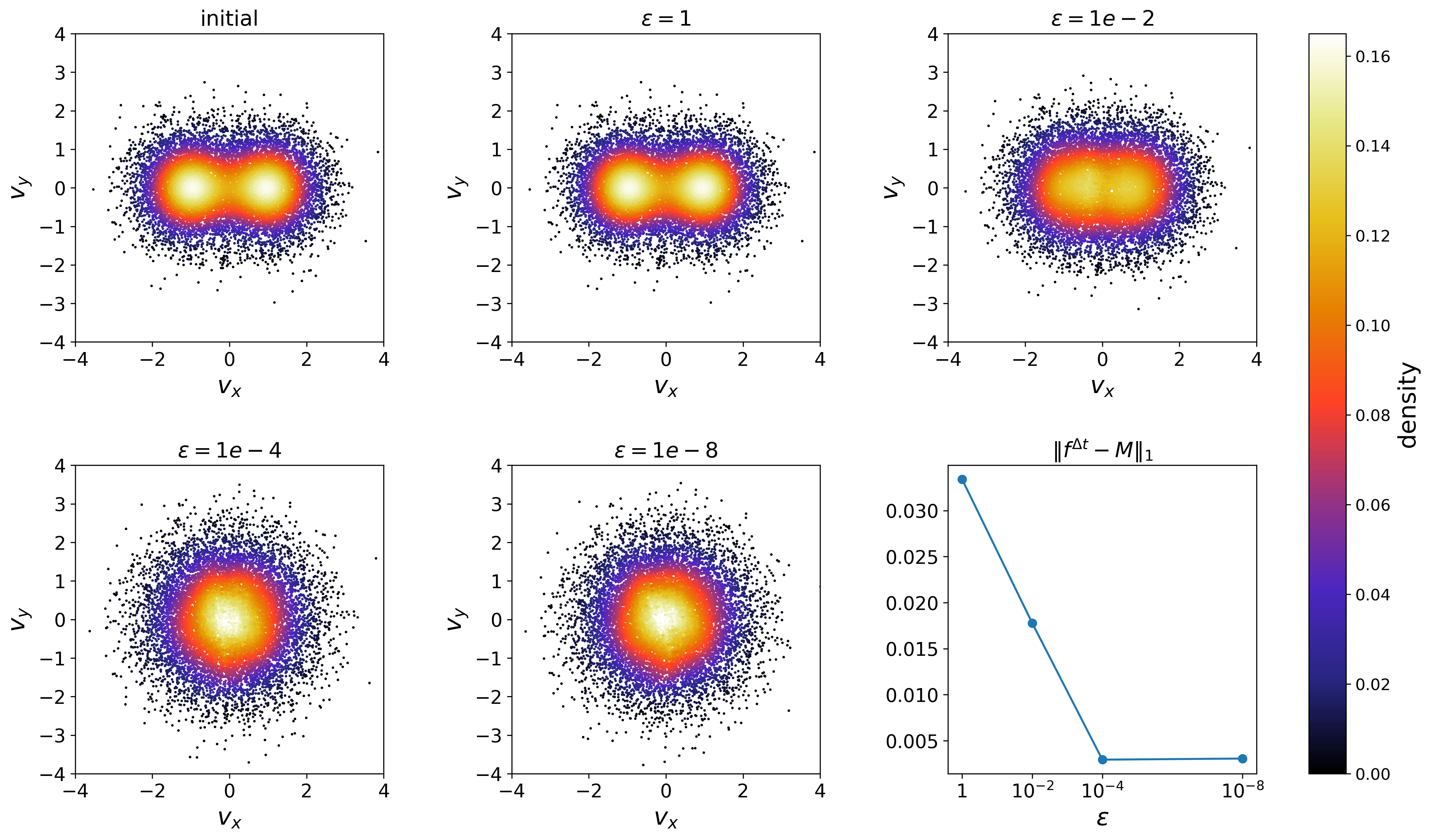}}
    \caption{Computed density $f^{\Delta t}$ for the Landau bi-Maxwellian example under various Knudsen numbers $\varepsilon$.}
    \label{fig:homoLD_f}
\end{figure}

\begin{figure}[ht]
    \centerline{\includegraphics[scale=0.33]{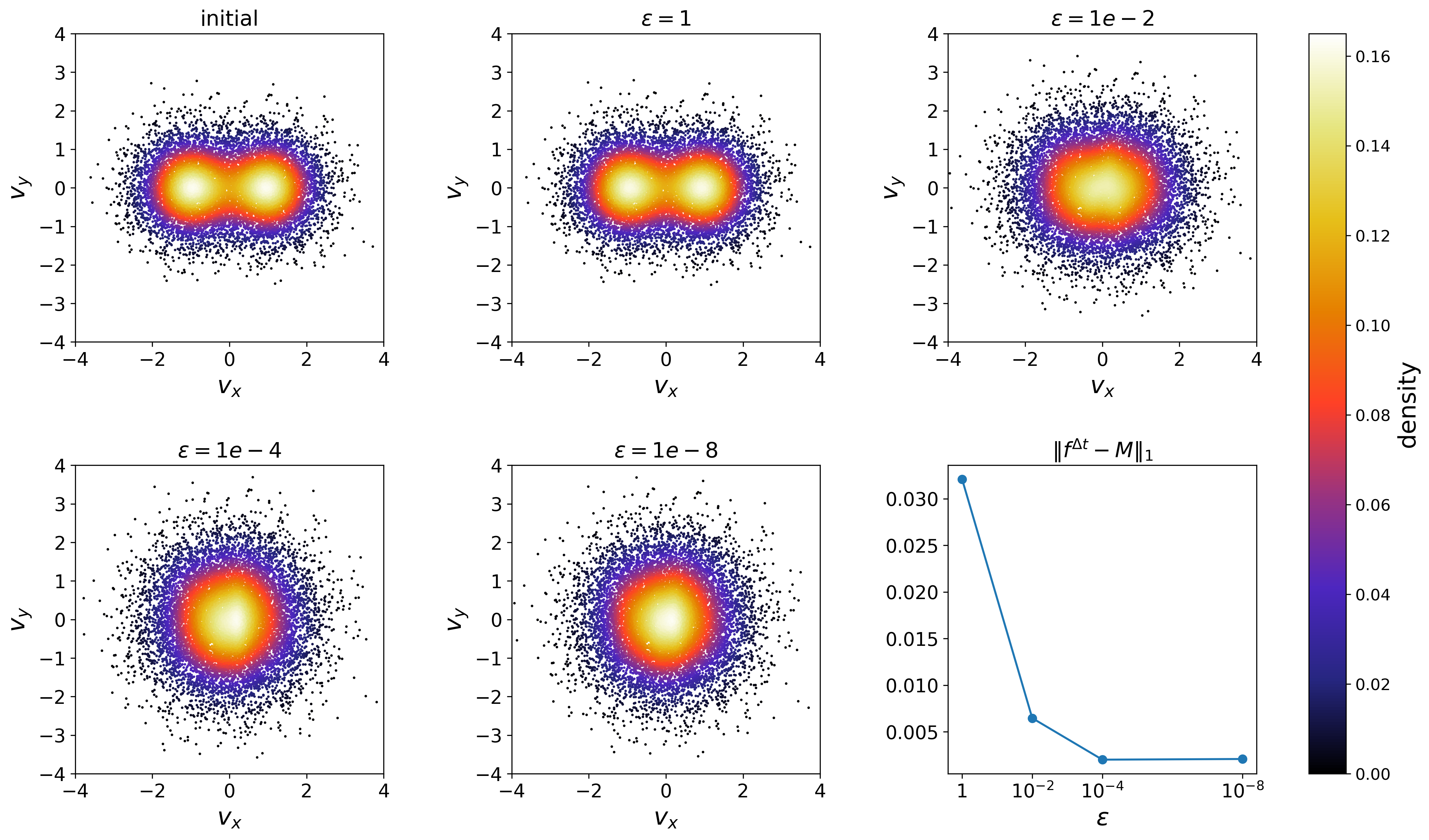}}
    \caption{Computed density $f^{\Delta t}$ for the Dougherty bi-Maxwellian example under various Knudsen numbers $\varepsilon$.}
    \label{fig:homoDougherty_f}
\end{figure}

\begin{table}[ht]
\small
\centering
\begin{tabular}{|c|c|c|c|c|}
    \hline
    & $\varepsilon = 1$ & $\varepsilon = 10^{-2}$ & $\varepsilon = 10^{-4}$ & $\varepsilon = 10^{-8}$ \\
    \hline
    \multicolumn{5}{c}{\textbf{Landau}} \\
    \hline
    momentum-x & 3.31426811e-09 & 1.45162296e-09 & 1.80800605e-09 & 5.99895764e-09 \\
    momentum-y & 6.85067661e-10 & 6.06702101e-09 & 4.67003714e-09 & 2.66412813e-09 \\
    energy (RK4) & 3.47823481e-09 & 3.47823481e-09 & 9.50196082e-07 & 2.14228898e-06 \\
    energy (imRK2) & 1.43667357e-07 & 9.47512220e-08 & 9.47512220e-08 & 9.47512220e-08 \\
    \hline
    \multicolumn{5}{c}{\textbf{Dougherty (projected-WGF)}} \\
    \hline
    momentum-x & 5.43274133e-09 & 4.03575746e-09 & 4.50141875e-09 & 7.76104776e-09 \\
    momentum-y & 2.52731651e-09 & 1.05111959e-08 & 1.59593934e-09 & 6.78590924e-09 \\
    energy (RK4) & 5.66743110e-08 & 5.66743110e-08 & 1.01034863e-06 & 7.71930048e-07 \\
    energy (imRK2) & 9.47512220e-08 & 1.43667357e-07 & 9.47512220e-08 & 9.47512220e-08 \\
    \hline
    \multicolumn{5}{c}{\textbf{Dougherty (WGF)}} \\
    \hline
    momentum-x & 4.35367446e-05 & 2.85005533e-06 & 9.95548662e-05 & 1.48417636e-04 \\
    momentum-y & 8.23002721e-06 & 4.04427095e-05 & 1.75976797e-04 & 1.03816261e-04 \\
    energy (RK4) & 8.95017184e-05 & 2.62697691e-02 & 3.84544710e-04 & 2.16746946e-04 \\
    \hline
\end{tabular}
\caption{Numerical errors of momentum and energy for the bi-Maxwellian example under various Knudsen numbers $\varepsilon$.}
\label{tab:homo_err}
\end{table}

\subsubsection{Relaxation of a discontinuous data (0D-3V)}
Now we consider a more challenging 3-dimensional example with discontinuous initial data
\begin{equation*}
    f^0(\bv) = 
    \begin{cases}
        \frac{\rho_1}{(2\pi T_1)^{3/2}} e^{-\frac{|\bv|^2}{2T_1}} \,, ~~ \text{for}~ \bv_x \leq 0 \,, \\
        \frac{\rho_2}{(2\pi T_2)^{3/2}} e^{-\frac{|\bv|^2}{2T_2}} \,, ~~ \text{for}~ \bv_x > 0 \,,
    \end{cases}
\end{equation*}
where $(\rho_1, T_1) = (\frac{16}{9}, 1)$ and $(\rho_2, T_2) = (\frac{2}{9}, \frac{1}{4})$. 

We set the Knudsen number $\varepsilon = \{1, 10^{-2}, 10^{-4} \}$, the time step size $\Delta t=0.01$, and the total number of particles $N=10^4$. The optimization hyperparameters are set as follows: $(\eta_{\max}, \eta_{\min}, T_0, T_{\max}, B) = (10^{-2}, 10^{-3}, 20, 100, 1000)$ for the Landau collision; $(\eta_{\max}, \eta_{\min}, T_0, T_{\max}) = (10^{-2}, 10^{-4}, 50, 200)$ for the Dougherty collision.

Fig.~\ref{fig:homoLD2_f} presents the histograms of the computed marginal density alongside the equilibrium marginal density for different Knudsen numbers $\varepsilon$, providing additional validation of the AP property of the proposed solver for the Landau and Dougherty collisions. We also present the momentum and energy error in Table~\ref{tab:homoLD2err}. Again, the momentum and the energy exhibit a very small numerical error.

\begin{figure}[ht]
    \centerline{\includegraphics[scale=0.35]{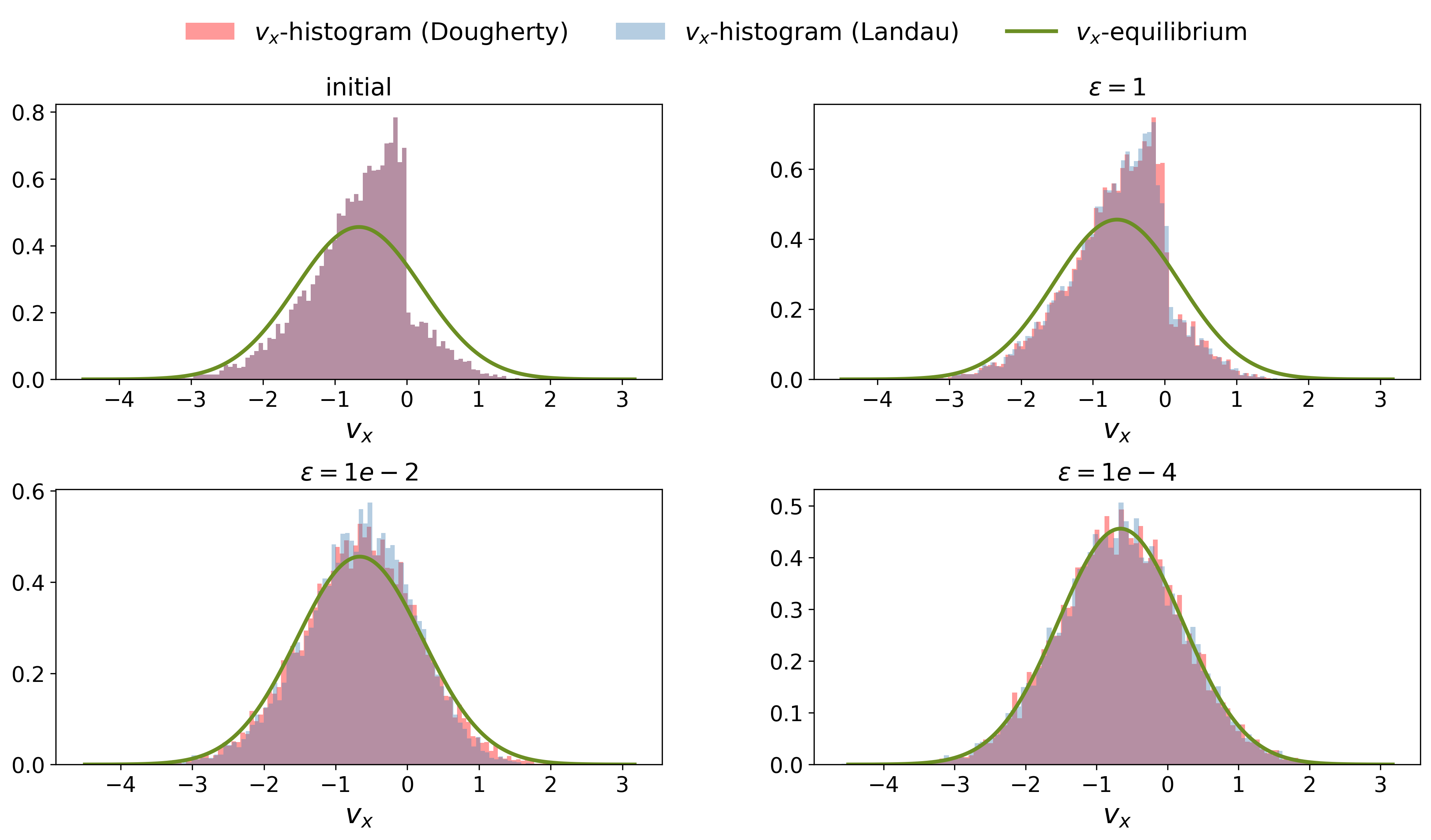}}
    \caption{Histograms of marginal densities for the discontinuous example under various Knudsen numbers $\varepsilon$.}
    \label{fig:homoLD2_f}
\end{figure}

\begin{table}[ht]
    \centering
    \small
    \begin{tabular}{ |c|c|c|c| }
        \hline
        & $\varepsilon=1$ & $\varepsilon=10^{-2}$ & $\varepsilon=10^{-4}$ \\ 
        \hline
        \multicolumn{4}{c}{\textbf{Landau}} \\
        \hline
        momentum-x & 9.38312661e-07 & 2.82661568e-07 & 1.11712659e-06 \\
        momentum-y & 1.19120407e-09 & 1.55726023e-08 & 2.12252665e-09 \\
        momentum-z & 2.83515353e-08 & 1.06564064e-08 & 1.36832047e-08 \\
        energy (RK4) & 4.02333118e-08 & 1.98185267e-07 & 4.57018631e-06 \\
        energy (imRK2) & 4.02333118e-08 & 4.02333118e-08 & 4.02333118e-08 \\
        \hline
        \multicolumn{4}{c}{\textbf{Dougherty (projected-WGF)}} \\
        \hline
        momentum-x & 4.42429889e-08 & 4.92198814e-07 & 1.65356840e-07 \\
        momentum-y & 1.23670750e-08 & 1.60276365e-09 & 1.42297201e-08 \\
        momentum-z & 3.54626288e-09 & 2.77606498e-08 & 2.71873821e-08 \\
        energy (RK4) & 1.98185267e-07 & 4.02333118e-08 & 7.90804642e-06 \\
        energy (imRK2) & 4.02333118e-08 & 4.02333118e-08 & 4.02333118e-08 \\
        \hline
    \end{tabular}
    \caption{Numerical errors of momentum and energy for the discontinuous example under various Knudsen numbers $\varepsilon$.}
    \label{tab:homoLD2err}
\end{table}

\subsubsection{Training efficiency}
We now assess the training efficiency of our method using the previous Landau bi-Maxwellian example. We first examine the benefit of evaluating \(\log |\det \nabla_{\bv}\T_1(\bv)|\) by quadrature rules over ODE solvers. The left panel of Figure~\ref{fig:runtime} shows the average per-epoch runtime using a 5-point Gauss--Legendre rule and a 6-stage RK4 method. Both implementations scale as \(\mathcal{O}(N)\) with respect to the number of particles, while the quadrature-based approach is approximately \(2.6\times\) faster. This improvement is practically significant for the inhomogeneous tests, where a single simulation may take several hours.

We also compare the runtime of the implicit midpoint and RK4 solvers for the particle dynamics discretization. The right panel of Figure~\ref{fig:runtime} shows the per-epoch runtime as a function of the Knudsen number \(\varepsilon\). The runtime of RK4 is essentially insensitive to \(\varepsilon\), whereas the implicit midpoint method is more expensive because each inner-time step requires solving a nonlinear system. Nevertheless, with a quasi-Newton solver and JFB training, the implicit method remains efficient in practice: the number of quasi-Newton iterations stays small, typically around five even in the stiff regime \(\varepsilon=10^{-5}\), while the training avoids cumbersome backward pass and large memory overhead.

Finally, Figure~\ref{fig:loss} illustrates the training behavior for different Knudsen numbers \(\varepsilon\). The left panel shows the stochastic training loss in Algorithm~\ref{alg:SGD}, while the right panel reports the corresponding true loss evaluated on all particles. Owing to random batching, the training loss continues to oscillate after the initial transient, especially for small \(\varepsilon\). Nevertheless, the full loss decreases rapidly and then stabilizes for all tested \(\varepsilon\), indicating that the stochastic optimization remains effective. Moreover, the full loss decays on essentially the same epoch scale across all values of \(\varepsilon\), suggesting that the optimization is only weakly affected by \(\varepsilon\).

\begin{figure}[ht]
    \centerline{\includegraphics[scale=0.4]{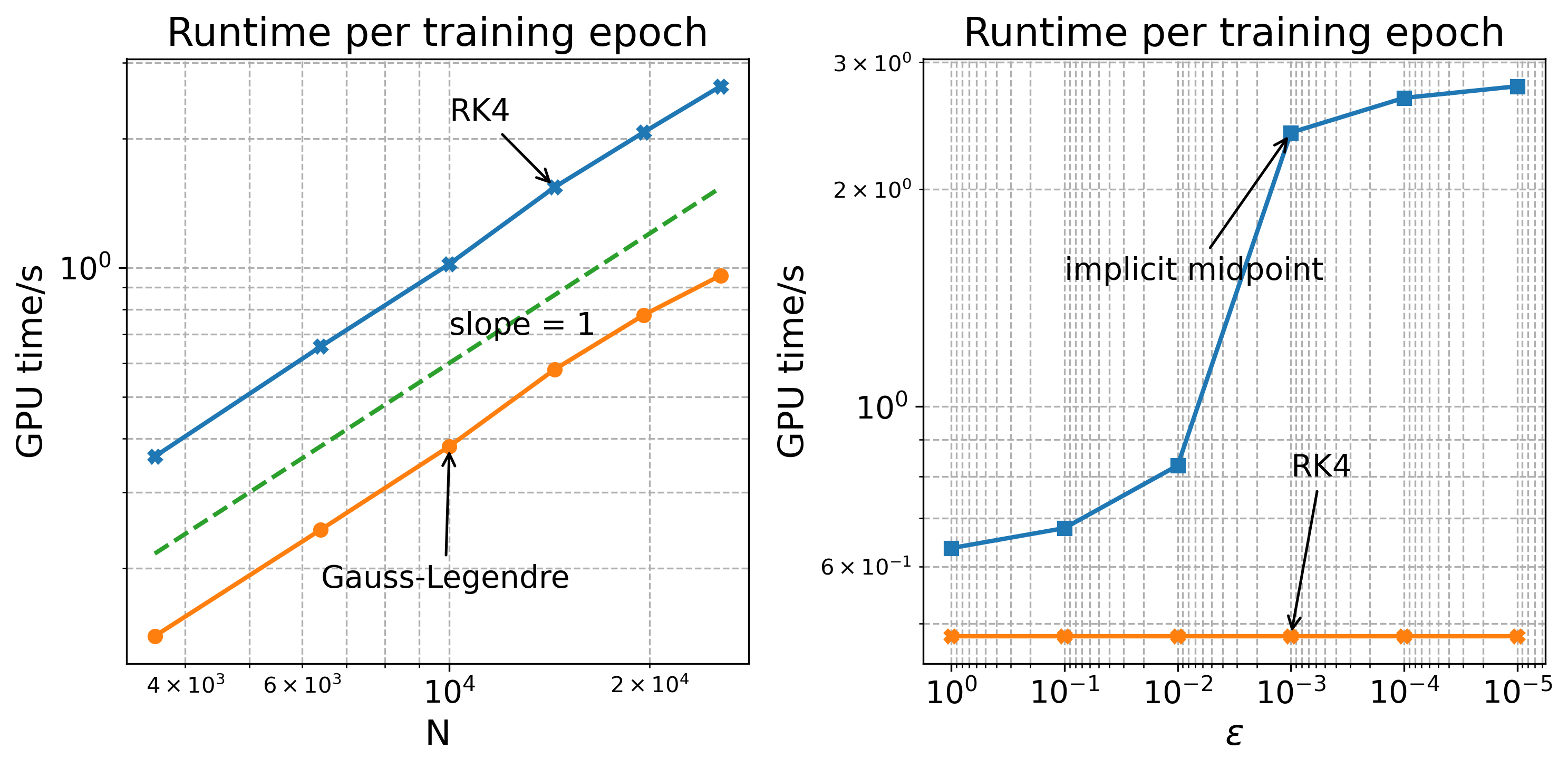}}
    \caption{Average GPU runtime per training epoch for the Landau bi-Maxwellian example. Left: comparison between computing $\log |\det \nabla_{\bv} \T_1(\bv)|$ by the RK4 solver and by the Gauss--Legendre quadrature rule as the particle number $N$ varies. Right: comparison between solving particle dynamics by the RK4 solver and by the implicit midpoint solver for various Knudsen numbers $\varepsilon$.}
    \label{fig:runtime}
\end{figure}

\begin{figure}[ht]
    \centerline{\includegraphics[scale=0.4]{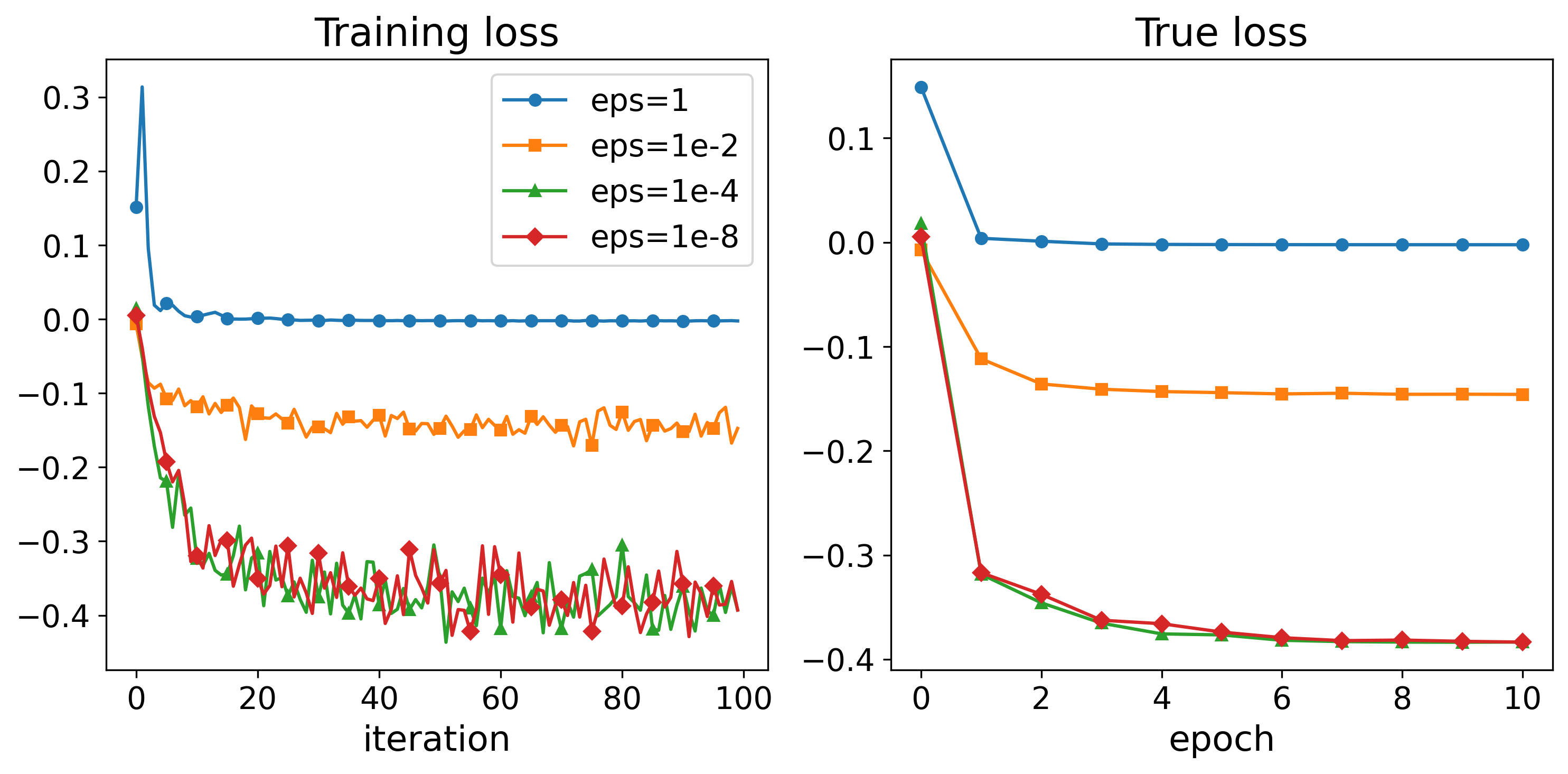}}
    \caption{Training behavior of the Landau bi-Maxwellian example under various Knudsen numbers $\varepsilon$. Left: stochastic mini-batch training loss versus epoch. Right: the corresponding full loss evaluated on all particles. Here $1 \operatorname{epoch} = 10 \operatorname{iterations}$.}
    \label{fig:loss}
\end{figure}

\subsection{Inhomogeneous Landau collision}
\subsubsection{Global conservation and entropy dissipation (1D-2V)}
We first check the global conservation and entropy dissipation property of our solver. The initial condition is given by 
\begin{equation*}
    f^0(\bx, \bv) = \frac{1}{2}\frac{\rho(x)}{2\pi} \left(e^{-\frac{(\bv_x - 1)^2}{2}} + e^{-\frac{(\bv_x + 1)^2}{2}} \right) e^{-\frac{\bv_y^2}{2}} \,,~~ \rho(x) = \frac{2+\sin(\pi x)}{3} \,.
\end{equation*}
The computational domain is chosen as $\Omega = [-1,1]$ and is divided into $N_c = 100$ disjoint cells. The Knudsen number is set to $\varepsilon = 1$ and the time step size is $\Delta t=0.1$. The total number of particles is $N=10^6$. For the optimization hyperparameters, we set $(\eta_{\max}, \eta_{\min}, T_0, T_{\max}, B) = (0.02, 0.01, 10, 40, 1500)$.

We plot the time evolution of the total energy error, total momentum error, and total entropy in Fig.~\ref{fig:noneq}. As expected, the total momentum is conserved, and the total energy exhibits only a very small error. The total entropy decreases over time.

\begin{figure}[ht]
    \centerline{\includegraphics[scale=0.33]{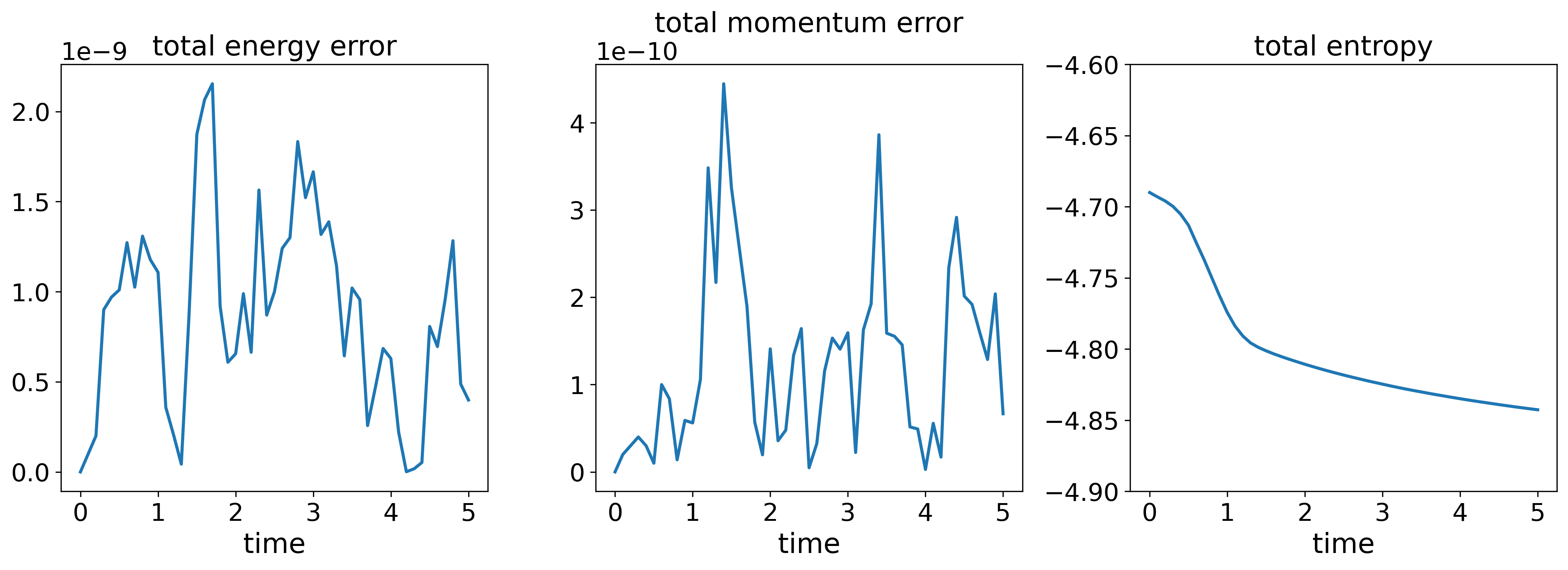}}
    \caption{Global conservation and entropy dissipation for an example with non-equilibrium initial data. Left: time evolution of total energy error. Center: time evolution of total momentum error. Right: time evolution of total entropy.}
    \label{fig:noneq}
\end{figure}

\subsubsection{Mixing regime (1D-2V)}
We then consider a mixed kinetic and fluid regimes where the Knudsen number $\varepsilon$ increases smoothly from a small $\varepsilon_0$ to $\mathcal{O}(1)$ and then jumps back to $\varepsilon_0$ in space:
\begin{equation*}
    \varepsilon(x) = 
    \begin{cases}
        \varepsilon_0 + \frac{1}{2} (\tanh(5-10x) + \tanh(5+10x)) \,, & x \leq 0.3 \,, \\
        \varepsilon_0 \,, & x > 0.3 \,,
    \end{cases}
\end{equation*}
with $\varepsilon_0 = 10^{-3}$. The picture of $\varepsilon$ is shown in the left plot of Fig.~\ref{fig:mix2}. The initial data is given by a local Maxwellian \eqref{local_eq} with $\rho(x) = \frac{2+\sin(\pi x)}{3}$, $\bu=(0.2,0)$, and $T(x) = \frac{3+\cos(\pi x)}{4}$. The periodic boundary condition in $x$ is applied. 

In this test, we divide the computational domain $\Omega = [-1,1]$ into $N_c = 100$ disjoint cells. We use a time step size of $\Delta t=0.01$ and $N=10^6$ particles, and set hyperparameters $(\eta_{\max}, \eta_{\min}, T_0, T_{\max}, B) = (10^{-2}, 10^{-3}, 10, 40, 1500)$. 

The center and right panels of Fig.~\ref{fig:mix2} show that, even when the Knudsen number $\varepsilon$ is spatially non-uniform, the total momentum is conserved and the total energy is preserved up to a very small error. We further compare our method with the penalization method \cite{JIN20116420}. Fig.~\ref{fig:mix} displays the local macroscopic quantities at $t=0.2$ and $t=0.4$, which are in close agreement with the reference solution.

\begin{figure}[ht]
    \centerline{\includegraphics[scale=0.33]{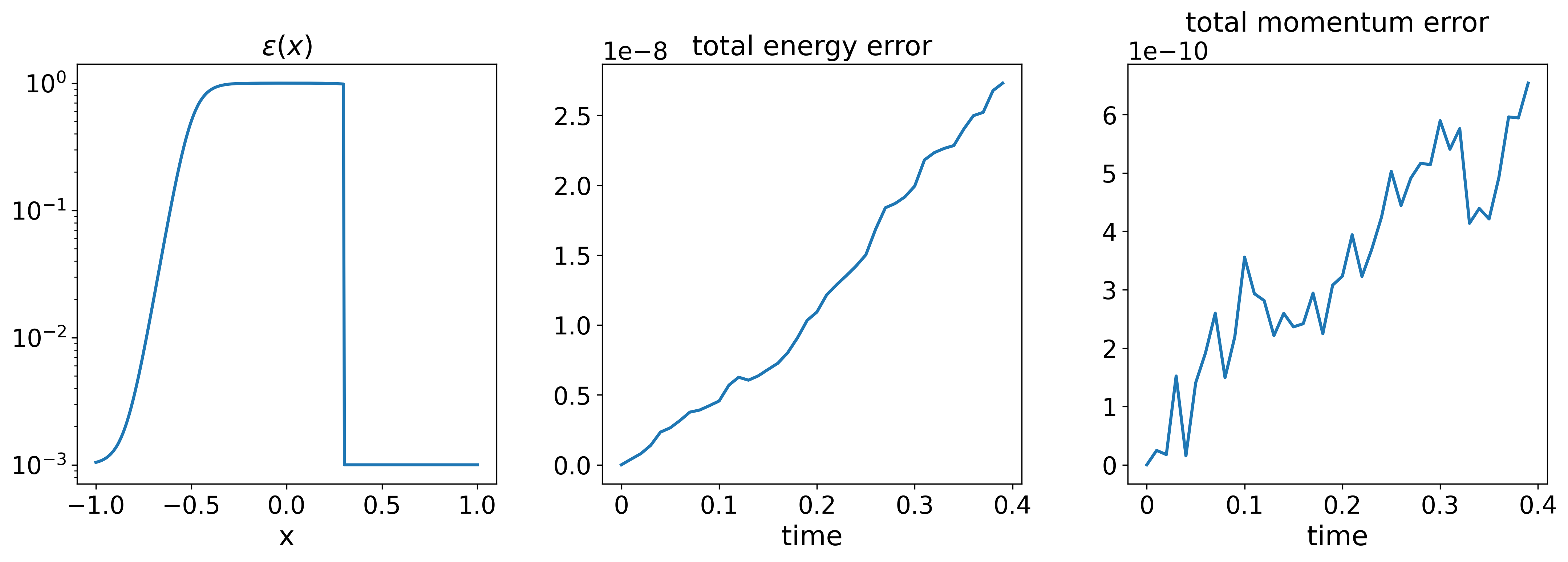}}
    \caption{A problem with mixing regime. Left: the spatially varying Knudsen number $\varepsilon$. Center: time evolution of total energy error. Right: time evolution of total momentum error.}
    \label{fig:mix2}
\end{figure}
\begin{figure}[ht]
    \centerline{\includegraphics[scale=0.38]{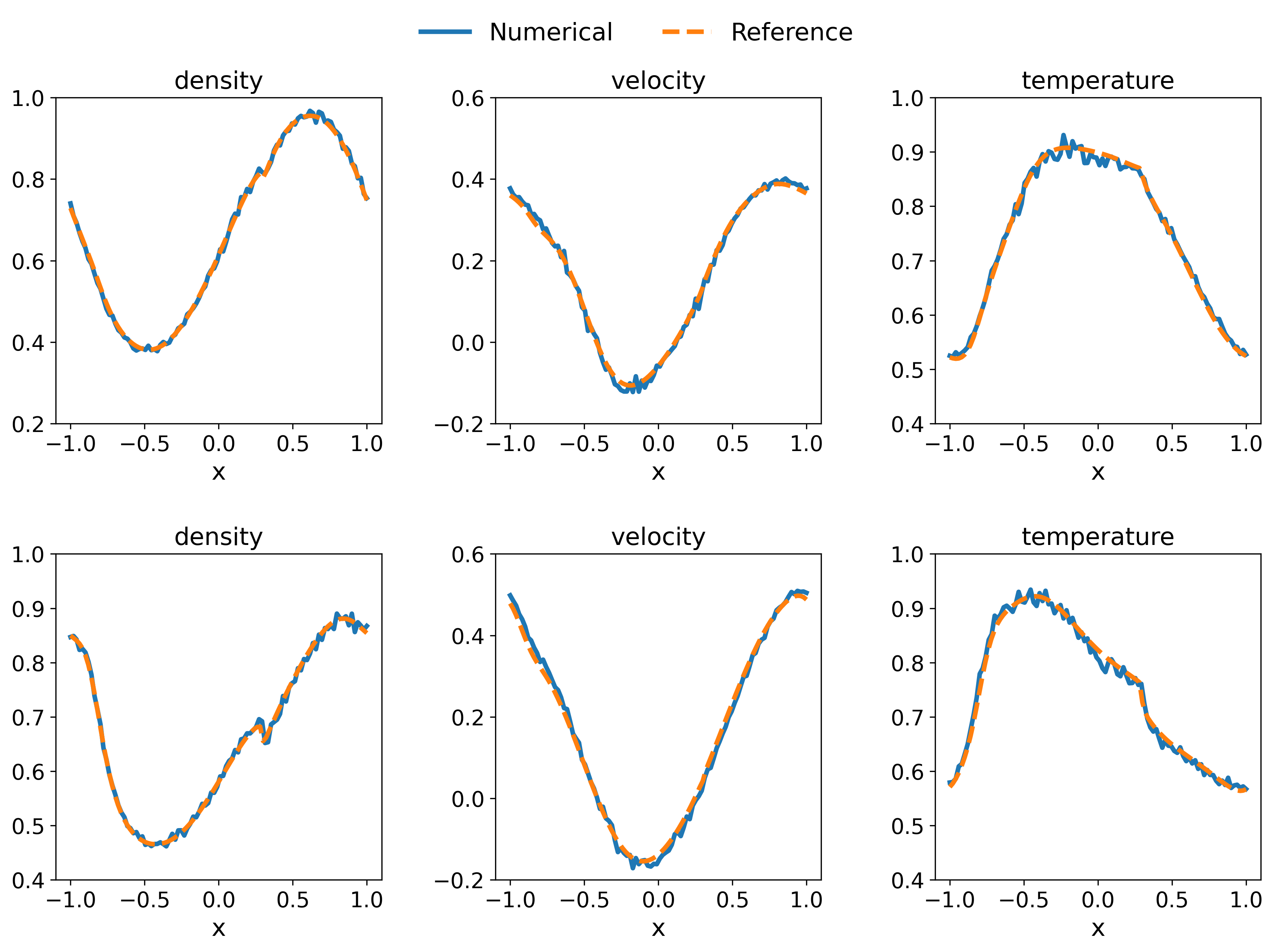}}
    \caption{The comparison of density, velocity, and temperature at $t=0.2$ (top) and $t=0.4$ (bottom) between the numerical solution and the reference solution by the penalization method.}
    \label{fig:mix}
\end{figure}

\subsubsection{The Riemann problem (1D-3V)}
Lastly, we simulate the Sod shock tube problem, a standard benchmark for assessing the ability to capture unsteady flow profiles and accurately resolve shock waves and contact discontinuities. We choose domain $\Omega = [0,1]$, where the initial condition is given by a local Maxwellian \eqref{local_eq} with $\bu=(0,0,0)$ and
\begin{equation*}
    \begin{cases}
        (\rho \,, T) = (1 \,, 1) \,,~~ \text{if}~ 0 \leq x \leq 0.5 \,, \\
        (\rho \,, T) = (\frac{1}{8} \,, \frac{1}{4}) \,,~~ \text{if}~ 0.5 < x \leq 1 \,.
    \end{cases}
\end{equation*}
The reflecting boundary condition in $x$ is applied.

In this test, we divide $\Omega$ into $N_c = 200$ disjoint cells. We use a time step size of $\Delta t=0.005$ and set the Knudsen number to $\varepsilon = 10^{-6}$. The total number of particles is $N=9 \times 10^5$. For the optimization hyperparameters, we set $(\eta_{\max}, \eta_{\min}, T_0, T_{\max}, B) = (0.01, 0.005, 10, 40, 1500)$. Fig.~\ref{fig:Riemann} presents the macroscopic results at $t=0.1$ and $t=0.2$, which converge to the hydrodynamic solution provided by the exact Riemann solver for the Euler system \eqref{Euler}.

\begin{figure}[ht]
    \centerline{\includegraphics[scale=0.38]{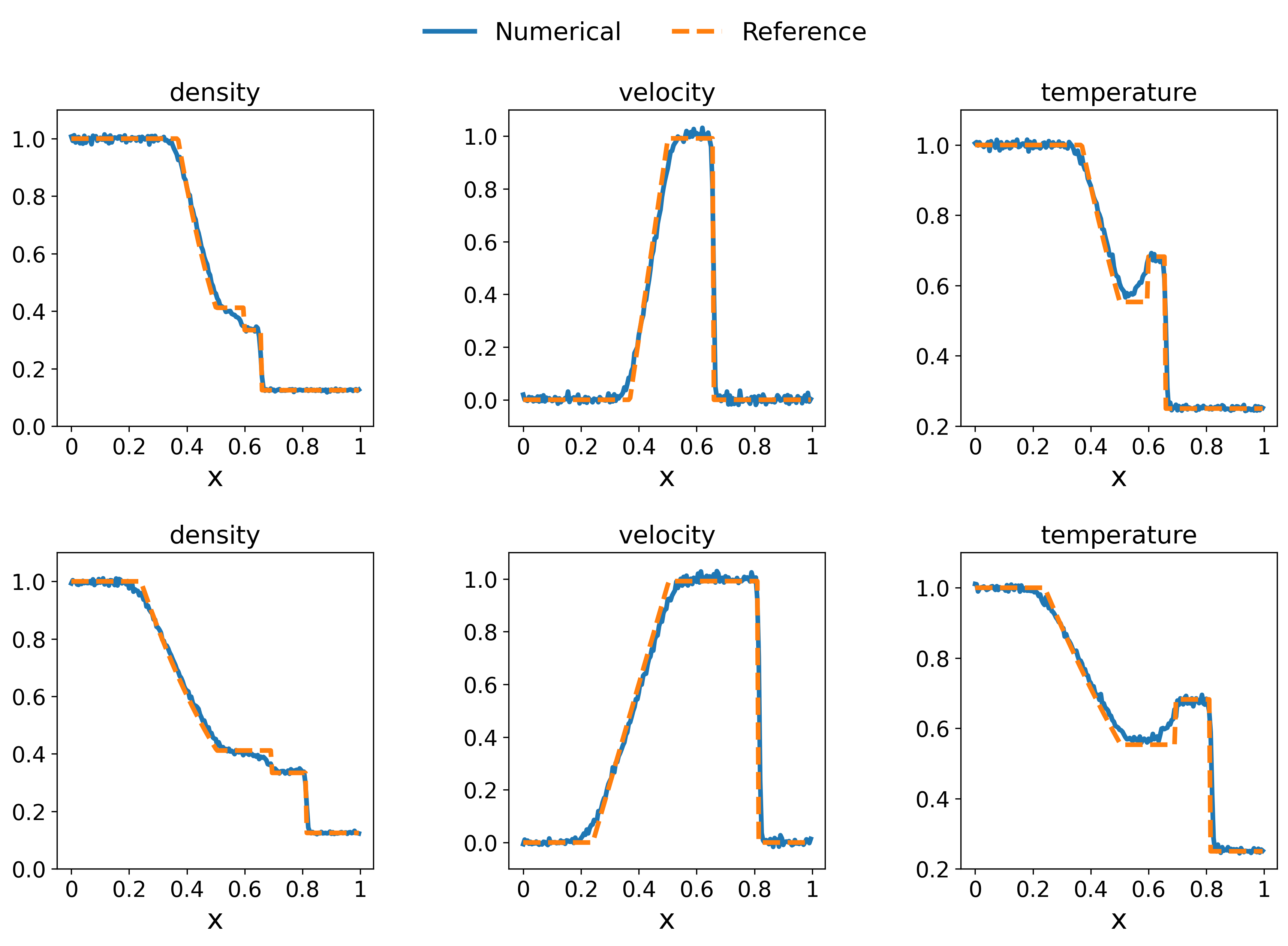}}
    \caption{The comparison of density, velocity, and temperature at $t=0.1$ (top) and $t=0.2$ (bottom) between the numerical solution and the exact Riemann solution for the Euler equations.}
    \label{fig:Riemann}
\end{figure}

\section{Conclusion} \label{sec:5}
In this work, we developed asymptotic-preserving deterministic particle methods for collisional plasma models under hydrodynamic scaling, covering both the Landau–Fokker–Planck and Dougherty collision operators. The key challenge in this regime is the treatment of stiff collision terms. In contrast to traditional grid-based approaches, which suffer from the curse of dimensionality, and existing particle methods, which are often stochastic or hybrid and introduce statistical noise, our approach provides a fully deterministic and structure-preserving framework.

Our method is built upon the gradient flow formulation of the collision operators combined with the minimizing movement (JKO) scheme. For the Landau collision operator, our approach may appear to be a direct extension of our previous work \cite{huang2024jkolandau}. However, we show that a critical difficulty arises from the presence of the small parameter, which affects the accurate evaluation of the entropy-dominated objective function in the fluid regime. To address this issue, we propose an inner-time quadrature strategy that improves both accuracy and efficiency. A similar treatment applies to the Dougherty collision operator; however, an additional subtlety arises: when formulated as a vanilla Wasserstein gradient flow, conservation of momentum and energy is lost at the particle level. To retain exact conservation, we instead propose a projected gradient flow formulation. Beyond algorithmic development, we also established connections between our variational framework and score-matching approaches, showing that both explicit and implicit score matching can be viewed as special cases of our formulation, while potentially facing limitations in stiff regimes. From a computational perspective, the integration of particle discretization, neural network parameterizations, and efficient training strategies enables our method to scale to high-dimensional problems.

Looking ahead, extending the framework to incorporate electromagnetic effects and multispecies interactions will be essential for realistic plasma simulations. In addition, while we employed an operator-splitting strategy in this work, the development of fully coupled (delocalized) JKO schemes remains an interesting direction for future research.

\section*{Acknowledgments}
We would like to thank Dr. Jiequn Han for insightful discussions. 
YH would also like to thank Dr. Howard Heaton and Prof. Samy Wu Fung for the discussion on Jacobian-free backpropagation. We acknowledge GPT-5.4 for its assistance in coding and refining the language and improving the clarity of the manuscript.

\appendix
\section{Energy loss in vanilla JKO for the Dougherty operator}\label{apdx:A}
Here we show that the vanilla JKO scheme \eqref{WGF_JKO} for the Dougherty operator preserves momentum but not energy. For simplicity, we assume $\rho^n=1$ and denote $\alpha=\Delta t / \varepsilon$. 

The dynamic JKO scheme \eqref{WGF_JKO} admits the following Monge formulation: given \(f^n\), we obtain \(f^{n+1}=\T\sharp f^n\) by solving
\begin{equation*}
    \inf_{\T} \int_{\R^{d_v}} |\T(\bv) - \bv|^2 f^n(\bv)\,\rd\bv + 2\alpha\, T^n \mathcal{H}\left(\T \sharp f^n \mid \M_{U^n} \right).
\end{equation*}
The corresponding optimality condition is 
\begin{equation}\label{apdx:a.1}
    \T(\bv)-\bv = -\alpha \left( T^n\nabla \log f^{n+1}(\T(\bv)) + \T(\bv)-\bu^n \right).
\end{equation}

We first prove conservation of momentum. Integrating both sides of \eqref{apdx:a.1} against \(f^n\), and using the identity $\int_{\R^{d_v}} \nabla \log f^{n+1}(\T(\bv)) f^n(\bv)\,\rd\bv = \int_{\R^{d_v}} \nabla f^{n+1} \,\rd\bv = 0$, we obtain $\bu^{n+1} - \bu^n = -\alpha\,(\bu^{n+1} - \bu^n)$, which implies that $\bu^{n+1} = \bu^n$.

Then we show that energy is not conserved. To this end, define the centered second moment $I^n:=\int_{\R^{d_v}} |\bv-\bu^n|^2 f^n \rd\bv$. Since $f^{n+1}=\T {\sharp} f^n$ and $\bu^{n+1} = \bu^n$, we have $I^{n+1} = \int_{\R^{d_v}} |\bv - \bu^n|^2 f^{n+1} \rd\bv = \int_{\R^{d_v}} |\T(\bv)-\bu^n|^2 f^n \rd\bv$. Hence
\begin{flalign*}
    I^{n+1} - I^{n}
    =& \int_{\R^{d_v}} (\T(\bv)-\bv) \cdot (\T(\bv)+\bv-2\bu^n) f^n \rd\bv \\
    =& \int_{\R^{d_v}} |\T(\bv)-\bv|^2 f^n \rd\bv + 2 \int_{\R^{d_v}} (\T(\bv)-\bv) \!\cdot\! (\bv-\bu^n) f^n \rd\bv \\
    =& - \underbrace{\int_{\R^{d_v}} |\T(\bv)-\bv|^2 f^n \rd\bv}_{\mathcal{I}_1} + 2 \underbrace{\int_{\R^{d_v}} (\T(\bv)-\bv) \!\cdot\! (\T(\bv) - \bu^n) f^n \rd\bv}_{\mathcal{I}_2} \,.
\end{flalign*}
Note that $\mathcal{I}_1 = W_2^2(f^{n+1}, f^n)$ and
\begin{flalign*}
    \mathcal{I}_2
    =& -\alpha \left( \int_{\R^{d_v}} T^n\nabla\log f^{n+1}(\T(\bv))\cdot (\T(\bv)-\bu^n) f^n \rd\bv + \int_{\R^{d_v}}|\T(\bv)-\bu^n|^2 f^n \rd\bv \right) \\
    =& -\alpha \left(\int_{\R^{d_v}} T^n\nabla f^{n+1} \cdot (\bv-\bu^n) \rd\bv + I^{n+1} \right) \\
    =& -\alpha \left(-d_v T^n + I^{n+1} \right) 
    = -\alpha \left(I^{n+1} - I^n \right) .
\end{flalign*}
Then $I^{n+1}=I^n-\frac{1}{1+2\alpha}\,W_2^2(f^{n+1},f^n)$. Since $I^n=2E^n-|\bu^n|^2$ and $I^{n+1}=2E^{n+1}-|\bu^n|^2$, we obtain the energy dissipation $E^{n+1}=E^n-\frac{1}{2(1+2\alpha)}\,W_2^2(f^{n+1},f^n)$.

\section{The geometry of the Dougherty operator}\label{apdx:B}
\paragraph{{Otto's Riemannian formalism}}
Let $\mathcal{P}_2$ denote the space of probability densities on $\R^{d_v}$ with finite second moment. In Otto's formalism \cite{villani2003topics}, $\mathcal{P}_2$ is viewed as an infinite-dimensional Riemannian manifold, with the 2-Wasserstein distance $W_2$ inducing a natural metric $g^W$ on its tangent bundle. More precisely, for each $f \in \mathcal{P}_2$, the tangent space at $f$, denoted by $\mathcal{T}_f\mathcal{P}_2$, is defined as
\begin{equation*}
    \mathcal{T}_f\mathcal{P}_2 := \left\{ - \nabla_{\bv} \cdot \bigl(f \nabla_{\bv}\phi\bigr) : \nabla_{\bv}\phi \in L_f^2 \right\}.
\end{equation*}
Given two tangent vectors $h_1, h_2 \in \mathcal{T}_f\mathcal{P}_2$, the metric $g^W$ is defined by
\begin{flalign*}
    g^W(h_1, h_2) = & \int_{\R^{d_v}} \nabla_{\bv} \phi_1(\bv) \cdot \nabla_{\bv} \phi_2(\bv) f(\bv) \rd\bv \,, \\
    & \text{where}~ h_i = - \nabla_{\bv} \cdot(f \nabla_{\bv}\phi_i) \,, \nabla_{\bv}\phi_i \in L^2_f ~\text{for}~ i=1,2 \,.
\end{flalign*}

\paragraph{Geometry of conservative subspace $\mathcal{P}_2^{con}$}
Any tangent vector $h \in \mathcal{T}_f \mathcal{P}_2^{con}$ must satisfy additional constraints $\int_{\R^{d_v}} h\bv \rd\bv = \int_{\R^{d_v}} h|\bv|^2 \rd\bv = 0$ due to the conservation of momentum and energy. Hence, the tangent space $\mathcal{T}_f\mathcal{P}_2^{con}$ is characterized as
\begin{equation*}
    \mathcal{T}_f\mathcal{P}_2^{con} = \left\{ - \nabla_{\bv} \cdot(f \nabla_{\bv}\phi) : \nabla_{\bv}\phi \in L^2_f \,, \int_{\R^{d_v}} \nabla_{\bv}\phi f \rd\bv = 0 \,, \int_{\R^{d_v}} \bv \cdot \nabla_{\bv}\phi f \rd\bv = 0 \right\} .
\end{equation*}
Since $\mathcal{T}_f \mathcal{P}_2$ endowed with the inner product $g^W$ is a Hilbert space, the closed subspace $\mathcal{T}_f \mathcal{P}_2^{con}$ admits a unique orthogonal complement space:
\begin{equation*}
    \mathcal{T}_f^{\perp} \mathcal{P}_2^{con} = \left\{ -\nabla_{\bv}\cdot (f \nabla_{\bv}\phi) : \phi \in\operatorname{span}\{\bv, |\bv|^2\} \right\} \,.
\end{equation*}

\paragraph{Projected gradient of entropy functional}
The Wasserstein gradient of the entropy $\mathcal{H}(f) = \int_{\R^{d_v}} f \log f \rd\bv$ is given by $\nabla_{W_2} \mathcal{H}(f) = -\nabla_{\bv} \cdot (f\nabla_{\bv} \tfrac{\delta \mathcal{H}}{\delta f})$. With the characterization of $\mathcal{T}_f^{\perp} \mathcal{P}_2^{con}$, we can express the projection of $\nabla_{W_2} \mathcal{H}(f)$ onto $\mathcal{T}_f \mathcal{P}_2^{con}$ as
\begin{equation*}
    \operatorname{Proj}_f ( \nabla_{W_2} \mathcal{H}(f) ) = -\nabla_{\bv} \cdot (f \nabla_{\bv} (\tfrac{\delta \mathcal{H}}{\delta f} - \phi)) \,, ~\text{with}~ \phi = \boldsymbol{a} \cdot \bv + b|\bv|^2 \,.
\end{equation*}
These coefficients $\boldsymbol{a} \in \R^{d_v}$ and $b \in \R$ are uniquely determined by the constraints for tangent vectors on $\mathcal{T}_f \mathcal{P}_2^{con}$,
\begin{equation*}
    \int_{\R^{d_v}} \nabla_{\bv} (\tfrac{\delta \mathcal{H}}{\delta f} - \phi) f \rd\bv = 0 \,,~~ \int_{\R^{d_v}} \bv \cdot \nabla_{\bv} (\tfrac{\delta \mathcal{H}}{\delta f} - \phi) f \rd\bv = 0 \,,
\end{equation*}
which yields $\boldsymbol{a} = \frac{\bu}{T}$ and $b=-\frac{1}{2T}$. Consequently, 
\begin{equation*}
    \operatorname{Proj}_f (\nabla_{W_2} \mathcal{H}(f)) = -\nabla_{\bv} \cdot (f (\nabla_{\bv} \log f + \tfrac{\bv-\bu}{T})) \,.
\end{equation*}

\section{Proofs for optimality conditions in Section~\ref{sec:2.2}}\label{apdx:C}
Here we provide a detailed derivation of the optimality conditions for the following variational problems appearing in Section~\ref{sec:2.2}. For simplicity, we denote $\alpha = \Delta t/\varepsilon$. 
\begin{flalign*}
    \inf_{\bs} ~ J_1[\bs] &= \int_{\R^{d_v}} \left[|\bs(\bv + \bs (\bv))|^2 - 2 \alpha (\nabla \!\cdot\! \bs)(\bv + \bs(\bv)) \right] f^n(\bv) \rd\bv \,, \\
    \inf_{\T} ~ J_2[\T] &= \int_{\R^{d_v}} \left( |\T(\bv) - \bv|^2 - 2\alpha \log|\det\nabla_{\bv} \T(\bv)| \right) f^n(\bv) \rd\bv \,, \\
    \inf_{\T} ~ J_3[\T] &= \int_{\R^{d_v}} \left(|\T(\bv) - \bv|^2 + 2\alpha \tr\left((\nabla_{\bv} \T(\bv))^{-1}\right)\right) f^n(\bv) \rd\bv \,.
\end{flalign*}
We first state a lemma which will be used later.
\begin{lemma}[Piola identity]\cite{Steinmann2015}\label{lem:piola}
Let $F$ be a $C^1$ vector field. Then
\begin{equation*}
    \nabla_{\bv} \cdot\big(\cof \nabla_{\bv}\T ~ F(\T(\bv))\big) = \det \nabla_{\bv}\T ~ (\nabla \cdot F)(\T(\bv)) \,.
\end{equation*}
\end{lemma}

\paragraph{Optimality condition of $J_1$}
Consider perturbation $\bs_{\eta}(\bv) = \bs(\bv) +\eta\bxi(\bv)$ with $\bxi\in C_c^\infty(\R^{d_v};\R^{d_v})$. Denote $\T(\bv) = \bv + \bs(\bv)$ and $f^{n+1} = \T \sharp f^n$. Note that 
\begin{flalign*}
    \tfrac{\rd}{\rd\eta} |_{\eta=0} ~ \bs_{\eta}(\bv + \bs_{\eta}(\bv)) 
    &= \bxi(\T(\bv)) + (\nabla \bs)(\T(\bv)) \cdot \bxi(\bv) \,, \\
    \tfrac{\rd}{\rd\eta} |_{\eta=0} ~ (\nabla \cdot \bs_{\eta})(\bv + \bs_{\eta}(\bv)) 
    &= (\nabla \cdot \bxi)(\T(\bv)) + (\nabla(\nabla \cdot \bs))(\T(\bv)) \cdot \bxi(\bv) \,.
\end{flalign*}
Hence
\begin{flalign*}
    0&=\tfrac{\rd}{\rd\eta} |_{\eta=0} ~ J_1[\bs_\eta] \\
    &=2\int_{\R^{d_v}} \big[ \bs(\T(\bv)) \cdot \tfrac{\rd}{\rd\eta}|_{\eta=0} ~ \bs_{\eta}(\bv + \bs_{\eta}(\bv)) - \alpha \tfrac{\rd}{\rd\eta}|_{\eta=0} ~ (\nabla \cdot \bs_{\eta})(\bv + \bs_{\eta}(\bv)) \big] f^n(\bv) \rd\bv \\
    &=2\int_{\R^{d_v}} [\bs(\T(\bv)) \cdot \bxi(\T(\bv)) - \alpha(\nabla \cdot \bxi)(\T(\bv))]f^n(\bv) \rd\bv \\
    & \qquad + 2 \int_{\R^{d_v}} [(\nabla\bs)(\T(\bv)) - \alpha (\nabla(\nabla \cdot \bs))(\T(\bv))]f^n(\bv) \cdot \bxi(\bv) \rd\bv \\
    &=2\int_{\R^{d_v}} [\bs(\bv) \cdot \bxi(\bv) - \alpha(\nabla \cdot \bxi)(\bv)]f^{n+1}(\bv) \rd\bv \\
    & \qquad + 2 \int_{\R^{d_v}} [(\nabla\bs)(\T(\bv)) - \alpha (\nabla(\nabla \cdot \bs))(\T(\bv))]f^n(\bv) \cdot \bxi(\bv) \rd\bv \\
    &= 2 \int_{\R^{d_v}} \!\! \left\{ \bs(\bv)f^{n+1}(\bv) + \alpha\nabla f^{n+1}(\bv) + [(\nabla\bs) - \alpha (\nabla(\nabla \cdot \bs))](\T(\bv)) f^n(\bv) \right\} \cdot \bxi(\bv) \rd\bv \,.
\end{flalign*}
Since $\bxi$ is arbitrary, we must have
\begin{equation*}
    \bs(\bv)f^{n+1}(\bv) + \alpha\nabla f^{n+1}(\bv) + [(\nabla\bs) - \alpha (\nabla(\nabla \cdot \bs))](\T(\bv)) f^n(\bv) = 0 \,.
\end{equation*}
Dividing both sides by $f^{n+1}$, we obtain 
\begin{equation*}
    \bs(\bv) = -\alpha\nabla\log f^{n+1}(\bv) + \left[\alpha (\nabla(\nabla \cdot \bs)) - (\nabla\bs) \right](\T(\bv)) \tfrac{f^n(\bv)}{f^{n+1}(\bv)} \,.
\end{equation*}

\paragraph{Optimality condition of $J_2$}
Consider perturbation $\T_\eta(\bv)=\T(\bv)+\eta\,\bxi(\bv)$ with $\bxi\in C_c^\infty(\R^{d_v};\R^{d_v})$. 
A direct computation of variations yields
$\frac{\rd}{\rd\eta} |_{\eta=0} ~ |\T_\eta(\bv)-\bv|^2 = 2(\T(\bv)-\bv)\cdot \bxi(\bv)$
and
\(\frac{\rd}{\rd\eta}|_{\eta=0} ~ \log|\det \nabla_{\bv} \T_\eta| = \tr\left((\nabla_{\bv} \T)^{-1}(\nabla_{\bv} \bxi)\right) = (\nabla_{\bv} \T)^{-\top} \!:\! \nabla_{\bv} \bxi \). 
Hence
\begin{flalign*}
    0=\tfrac{\rd}{\rd\eta} |_{\eta=0} ~ J_2[\T_\eta]
    &=2\int_{\R^{d_v}}\left[(\T(\bv)-\bv)\cdot\bxi-\alpha(\nabla_{\bv}\T)^{-\top}:\nabla_{\bv}\bxi\right]f^n\,\rd\bv \\
    &=2\int_{\R^{d_v}}\left[(\T(\bv)-\bv)f^n+\alpha\nabla_{\bv}\cdot\big((\nabla_{\bv}\T)^{-\top}f^n\big)\right]\cdot\bxi\,\rd\bv \,,
\end{flalign*}
where we use the integration by parts in the last equality. Since $\bxi$ is arbitrary, then 
\begin{equation*}
    (\T(\bv)-\bv)f^n+\alpha\nabla_{\bv} \cdot \big((\nabla_{\bv}\T)^{-\top}f^n\big) = 0 \,.
\end{equation*}
By Brenier's theorem \cite{villani2003topics}, $\T = \nabla_{\bv}\phi$ for some convex function $\phi$, and $\det\nabla_{\bv}\T = \det \nabla_{\bv}^2 \phi \geq 0$. Recall that $(\nabla_{\bv}\T)^{-\top}=\frac{\cof\nabla_{\bv}\T}{\det\nabla_{\bv}\T}$ and the change of variable formula $f^n(\bv) = f^{n+1}(\T(\bv)) \det\nabla_{\bv} \T$. It then follows that
\begin{equation*}
    (\T(\bv)-\bv)f^n+\alpha \nabla_{\bv}\cdot\big(\cof\nabla_{\bv}\T \, f^{n+1}(\T(\bv))\big) = 0 \,.
\end{equation*}
Apply Lemma~\ref{lem:piola} with $F=f^{n+1}$ gives 
\begin{equation*}
    (\T(\bv)-\bv)f^n + \alpha \det\nabla_{\bv}\T ~ (\nabla f^{n+1})(\T(\bv)) = 0 \,.
\end{equation*}
Use the change of variable formula again, we obtain
\begin{equation*}
    (\T(\bv)-\bv) + \alpha(\nabla\log f^{n+1})(\T(\bv)) = 0 \,.
\end{equation*}

\paragraph{Optimality condition of $J_3$}
Note that $\tfrac{\rd}{\rd\eta}|_{\eta=0} ~ \tr\left((\nabla_{\bv} \T_\eta)^{-1}\right)
= \\ - \tr\left((\nabla_{\bv} \T)^{-1}(\nabla_{\bv} \bxi)(\nabla_{\bv} \T)^{-1}\right) 
= - \left((\nabla_{\bv} \T)^{-\top}(\nabla_{\bv} \T)^{-\top}\right) \!:\! \nabla_{\bv} \bxi$.
Then
\begin{flalign*}
    0=\tfrac{\rd}{\rd\eta}|_{\eta=0} ~ J_3[\T_\eta]
    &=2\int_{\R^{d_v}}\Big[(\T(\bv)-\bv)\cdot\bxi
    -\alpha\big((\nabla_{\bv}\T)^{-\top}(\nabla_{\bv}\T)^{-\top}\big):\nabla_{\bv}\bxi\Big]f^n\,\rd\bv \\
    &=2\int_{\R^{d_v}}\Big[(\T(\bv)-\bv)f^n+\alpha\nabla_{\bv}\cdot\big((\nabla_{\bv}\T)^{-\top}(\nabla_{\bv}\T)^{-\top}f^n\big)\Big]\cdot\bxi\,\rd\bv \,,
\end{flalign*}
where we use the integration by parts in the last equality. Since $\bxi$ is arbitrary, then 
\begin{equation*}
    (\T(\bv)-\bv)f^n+\alpha\nabla_{\bv}\cdot\big((\nabla_{\bv}\T)^{-\top}(\nabla_{\bv}\T)^{-\top}f^n\big) = 0 \,.
\end{equation*}

Now we assume $\det\nabla_{\bv}\T > 0$. This assumption is natural, as it is also satisfied by the true optimal transport map. Arguing as before, we obtain
\begin{equation*}
    (\T(\bv)-\bv)f^n+\alpha\nabla_{\bv} \cdot\Big(\cof\nabla_{\bv}\T \, (\nabla_{\bv}\T)^{-\top} f^{n+1}(\T(\bv))\Big) = 0 \,.
\end{equation*}
Note that $(\nabla_{\bv}\T(\bv))^{-\top}= ((\nabla \T^{-1})(\T(\bv)))^\top$. Therefore, applying Lemma~\ref{lem:piola} with $F=(\nabla \T^{-1})^\top f^{n+1}$, yields
\begin{equation*}
    (\T(\bv)-\bv)+\alpha\tfrac{\nabla\cdot((\nabla \T^{-1})^\top f^{n+1})}{f^{n+1}} (\T(\bv)) = 0 \,.
\end{equation*}
Expanding the gradient computation, we obtain
\begin{equation*}
    (\T(\bv)-\bv)+\alpha \left((\nabla \T^{-1})^\top \nabla\log f^{n+1} + \Delta \T^{-1} \right) (\T(\bv)) = 0 \,.
\end{equation*}

\bibliographystyle{siamplain}
\bibliography{ref}
\end{document}